\documentclass[article,11pt,leqno]{article}
\usepackage{articlemacro}

\newcommand{\AffSch}{\textup{AffSch}}
\newcommand{\ft}{\textup{ft}}
\newcommand{\DM}{\textup{DM}}
\newcommand{\PreStk}{\textup{PreStk}}
\newcommand{\Stk}{\textup{Stk}}

\newcommand{\DGCat}{\textup{DGCat}_{\textup{cont}}}
\newcommand{\PrStL}{\textup{Pr}^{\textup{L}}_{\textup{st}}}
\newcommand{\Ind}{\textup{Ind}}

\DeclareMathOperator{\cofib}{\textup{cofib}}
\DeclareMathOperator{\Th}{\textup{Th}}

\newcommand{\ins}{\textup{ins}}
\newcommand{\BT}{\textup{BT}}
\newcommand{\Disp}{\textup{Disp}}
\newcommand{\quot}[1]{\left[#1\right]}
\newcommand{\PrLO}{\textup{Pr}^{L,\otimes}}

\newcommand{\Corr}[2]{\textup{Corr}({#1})_{#2}}
\newcommand{\lft}{\textup{lft}}
\newcommand{\CAlg}{\textup{CAlg}}

\date{}
\title{Motives on pro-algebraic stacks}
\author{Can Yaylali}%\href{mailto:math.yaylali@gmail.com}{math.yaylali@gmail.com}

\begin{document}
\maketitle
\begin{abstract}
We define an $\infty$-category of rational motives for inverse limits of algebraic stacks, so-called pro-algebraic stacks. We show that it admits a $6$-functor formalism for certain classes of morphisms.
On pro-schemes, we show that this 6-functor formalism is in the sense of Liu-Zheng. 
This theory yields an approach to the theory of motives for non-representable algebraic stacks and non-finite type morphisms.
\end{abstract}

\thispagestyle{empty}
\tableofcontents

\section{Introduction}
\vspace{6pt}\noindent\textbf{Motivation:}  Barsotti--Tate groups have been central to many studies in arithmetic geometry. Consider the stack $\BT$ of Barsotti--Tate groups over $\Spec(\FF_{p})$. Geometrically, $\BT$ is difficult to understand, as it is not an algebraic stack. However, we can express $\BT$ as a pro-system of algebraic stacks, $\BT \simeq \lim_{n \geq 0} \BT_{n}$, where $\BT_{n}$ denotes the algebraic stack of $n$-truncated Barsotti--Tate groups. Given the importance of $\BT$, it is natural to ask about its “motive” in $\DM(\FF_{p}, \QQ)$. Since $\BT$ is not algebraic, this question does not a priori make sense, though one can formally define
$$
	M(\BT) \coloneqq \colim_{n \geq 0} M(\BT_{n}).
$$
Alternatively, we could formally extend $\DM^{!}_{\QQ}$ to arbitrary prestacks and define $M(\BT)$ as the $!$-pull/push of the unit. However, there is no general reason why these two definitions should agree. This issue is often referred to as continuity. If $S = \lim_{i} S_{i}$ is a projective system of noetherian finite-dimensional schemes with affine transition maps, then
$$
	M(S) \cong \colim_{i} M(S_{i})
$$
provided that $S$ is also noetherian of finite dimension \cite[Thm.~14.3.1]{CD}.

Returning to $\BT$, both of the above definitions seem reasonable. However, the formal extension $\DM^{!}_{\QQ}$ does not come with the six-functor formalism one would typically expect in a theory of motives. For computations, such a formalism would be very useful. This becomes clearer when considering the stack $\Disp$ of so-called displays \cite{Lau}. This stack also admits a presentation $\Disp \simeq \lim_{n \geq 0} \Disp_{n}$ as a pro-system of algebraic stacks over $\Spec(\FF_{p})$, and we have a map of pro-systems of algebraic stacks $\phi\colon \BT \to \Disp$ \cite{Lau}. This morphism is level-wise smooth and a universal homeomorphism. The motive of each $\Disp_{n}$ is in fact Tate (cf. Theorem~\ref{thm-Disp}), and hence easier to understand than the motive of each $\BT_{n}$. To transfer information from $\Disp$ to $\BT$, a definition of the motive of a pro-system of algebraic stacks should be compatible with morphisms, at least those of pro-(locally of finite type).

In this article, we will investigate such situations more closely. To be more precise, we will study the category of rational motives on pro-systems of algebraic stacks $X \simeq \lim_{i} X_{i}$ defined via $\colim_{i} \DM(X_{i})$, and investigate its functorial properties.

%Let us be a bit more precise. Let $\Dcal$ be a compactly generated $6$-functor formalism on algebraic stacks. When $x\rightarrow y$ is a morphism of pro-systems of algebraic stacks locally of finite type over $K$, does the induced functor 
%$$
%	\colim_{\Ical}\Dcal \circ y\rightarrow  \colim_{\Ical} \Dcal \circ x
%$$
%extend to a full $6$-functor formalism?

\vspace{6pt}\noindent\textbf{Motives on algebraic stacks:} We can define $\DM(F)$, the DG-category of motives for any prestack $F\colon \Ring\rightarrow \SS$ (here $\SS$ denotes the $\infty$-category of Kan-complexes) via right Kan extension along the Yoneda embedding \cite{RS1}. In this way, we can define $\DM(X)$ for any algebraic stack locally of finite type over a field $K$ for rational motives. Since rational motives satisfy \'etale descent, the $\infty$-category of rational motives on an algebraic stack $X$ with smooth atlas $\Xtilde$ can be computed via the formula
$$
\begin{tikzcd}
	 \DM(X)\coloneqq \lim ( \DM(\Xtilde)\arrow[r,"",shift left = 0.3em]\arrow[r,"",shift right = 0.3em]&\DM(\Xtilde\times_{X}\Xtilde)\arrow[r,"",shift left = 0.6em]\arrow[r,""]\arrow[r,"",shift right = 0.6em]&\dots )
\end{tikzcd}
$$
\cite{RS1}. The problem with these definitions is that the computation of the motive of a limit $\lim_{i} X_{i}$ as above is not straightforward. To circumvent this problem, we will work with an alternative definition and show that in classical situations our definition agrees with the usual one.

\vspace{6pt}\noindent\textbf{Motives on pro-algebraic stacks:} 
Let us fix a filtered $\infty$-category $\Ical$.

\begin{intro-definition}
\label{def.intro}
Let $K$ be a field and $x\colon \Ical\rightarrow \PreStk_{K}^{\op}$ be an inverse system of prestacks. Assume that each $X_{i}\coloneqq x(i)$ is representable by an algebraic stack locally of finite type over $K$ and let us denote by $X\coloneqq \lim x$. Then we call the tuple $(x,X)$ a \textit{pro-$\Ical$-algebraic stack}. 

We call $(x,X)$ \textit{tame}, if $x_{ij}$ is smooth for all $i\to j\in\Ical$.
\end{intro-definition}

Note that there is an obvious notion of morphisms of pro-$\Ical$-algebraic stacks.

\begin{intro-definition}
Let $(x,X)$ be a pro-$\Ical$-algebraic stack. Then we define the $\infty$-category of rational motives on $(x,X)$ to be 
$$
	\DM(x,X)\coloneqq \colim_{i\in \Ical,x^{*}}\DM(X_{i}), 
$$
where the colimit is computed in $\DGCat\coloneqq \mathrm{Mod}_{\Dcal(\QQ)}(\Pr^{L}_{\mathrm{st}})$.
\end{intro-definition}

With this construction, we will obtain a full $6$-functor formalism on pro-$\Ical$-algebraic stacks, which is not provided by the right Kan extended version. The caveat is that this $6$-functor formalism does not exist for all morphisms of pro-$\Ical$-algebraic stacks. We have to make restrictions depending if $\Ical\simeq \NN$, the poset $(\NN,\leq)$ viewed as an $\infty$-category, or not. We will explain this in the main text of this article.

To keep this introduction light, we will only state a summarized version of the existence of a (motivic) $6$-functor formalism and refer for a precise version to the main text.

\begin{intro-theorem}[\protect{\ref{thm.6-ff}}]
\label{thm.intro.2}
	The $\infty$-category $\DM(x,X)$ is closed symmetric monoidal, compactly generated and the assignment $(x,X)\mapsto \DM(x,X)$ is contravariantly functorial in $(x,X)$. 
		
	Let $f\colon (x,X)\rightarrow (y,Y)$ be a morphisms of tame pro-$\NN$-algebraic stacks that is levelwise locally of finite type.   Then $f^{*}\coloneqq \DM(f)$ is part of a $6$-functor formalism if we assume that the exchange map $f_{n+1!}x_{n}^{!}\rightarrow y_{n}^{!}f_{n!}$ induced by the diagram is 
	$$
		\begin{tikzcd}
			X_{n+1}\arrow[r,"x_{n}"]\arrow[d,"f_{n+1}", swap]& X_{n}\arrow[d,"f_{n}"]\\
			Y_{n+1}\arrow[r,"y_{n}"]&Y_{n}
		\end{tikzcd}
	$$
 an equivalence for all $n\in\NN$. Moreover, if $f_{n}^{*}y_{n*}\to x_{n*}f_{n+1}^{*}$  is an equivalence, then $f^{*}$ defines a pullback formalism with respect to smooth maps.
 
Further, this $6$-functor formalism satisfies homotopy invariance, T-stability, localization and purity in a suitable sense.\par 
 Finally, if we assume that each square above is cartesian, the same results hold for any locally of finite type morphism of pro-$\Ical$-algebraic stacks for any filtered $\infty$-category without any assumption on tameness and the exchange map.
\end{intro-theorem} 

The assumption on the base change map is satisfied, for example, if $(x,X)$ and $(y,Y)$ are tame and square above is cartesian. Surprisingly, the condition on $f_{n+1!}x_{n}^{!}\rightarrow y_{n}^{!}f_{n!}$ only appears if the square above is not cartesian, since otherwise, we can use the $!/*$-base change equivalence. More importantly, we will see that it also holds for the structure map $(\tau, \Disp) \to \Spec(\FF_{p})$, (where $\tau_{n} \colon \Disp_{n+1} \to \Disp_{n}$ denotes the truncation map on displays), for which each square is not cartesian. In particular, we can define the motive of $\Disp$ via $!$-pull/push of the unit. Moreover, we will see that in the special case of $\Disp$, we have continuity
$$
	\DM(\Disp) \simeq \DM(\tau, \Disp)
$$
(see~Proposition~\ref{prop.Disp.motive}). By Theorem~\ref{thm.intro.2}, we obtain adjoint functors $	 p_{\sharp}\dashv p^{*}\dashv p_{*}$ and $p_{!}\dashv p^{!}$ with the expected compatibilities for the structure map $p\colon \Disp\to\Spec(\FF_{p})$. These do not follow formally from the definition of $\DM(\Disp)$ via Kan extension.

The definition via $ \DM(\tau, \Disp)$ allows us to compute $M(\Disp)$ more explicitly.  

\begin{intro-theorem}[\protect{\ref{thm-Disp}}]
	The pullback along the truncation map $\Disp\rightarrow \Disp_{1}$ induces a fully faithful embedding.
	$$
		\DM(\Disp_{1})\hookrightarrow \DM(\Disp)
	$$
	In particular, we have an equivalence of motives 
	$$
		M(\Disp)\simeq M(\Disp_{1})
	$$
	inside $\DM(\FF_{p})$ and moreover it is contained in the full stable cocomplete subcategory of $\DM(\FF_{p})$ generated by Tate motives.
\end{intro-theorem} 

As a consequence, working with pro-algebraic stacks, we see that $H^{*,*}(BT,\QQ)\cong H^{*,*}(\Disp_{1},\QQ)$ (see Corollary \ref{cor.BT}).

\vspace{6pt}\noindent\textbf{Generalization:} 
Since most arguments in the proof of Theorem \ref{thm.intro.2} apply in a more general context, we will frequently work with compactly generated $6$-functor formalisms. Although our primary interest lies in rational motives, we will formulate the necessary statements as generally as possible. This broader perspective encompasses both compactly generated pullback formalisms \cite{DG} and $6$-functor formalisms in the sense of Liu–Zheng \cite{LiuZheng}. The advantage of this generality is that the key computations remain valid in a wider range of settings, such as pro-systems of ind-schemes. The \'etale derived category of such systems, for instance, is considered by Bouthier, Kazhdan and Varshavsky \cite{Kaz}.

\subsection*{Structure of this article}
We start by defining pro-$\Ical$-algebraic stacks and listing some examples. We define the $\infty$-category of rational motives on pro-$\Ical$-algebraic stacks.\par Afterward, we prove, that on certain classes of morphisms, this definition yields a (motivic) $6$-functor formalism.\par 
Lastly, we will have a closer look at motivic cohomology for pro-$\Ical$-algebraic stacks and apply our results to the stack of displays and arbitrary Galois extensions.

\subsection*{Acknowledgement}
This project started after a talk with Torsten Wedhorn about motives on limits of algebraic stacks and I would like to thank him for various discussions and his interest in this project. I want to thank Chirantan Chowdhury for his comments about correspondences and the idea to use the formalism of bisimplicial sets to extend diagrams to correspondences. I also thank the anonymous referee for all their suggestions. Further, I would like to thank R\i zacan \c{C}ilo\u{g}lu, Martin Gallauer, Benjamin Hennion, Ryomei Iwasa, Hyungseop Kim, Timo Richarz, Jakob Scholbach for helpful
discussions. \par
This project was funded by the Deutsche Forschungsgemeinschaft (DFG, German Research Foundation) - project number 524431573, the Deutsche Forschungsgemeinschaft (DFG, German Research Foundation) TRR 326 \textit{Geometry and Arithmetic of Uniformized Structures}, project number 444845124.

\subsection{Assumptions and notations}
\label{notations}
In the following, we want to fix some categorical and algebraic notation that is used throughout this article.
\subsubsection*{Categorical notation}
Throughout, we fix some inaccessible regular cardinal $\kappa$. By \textit{small}, we will mean \textit{$\kappa$-small}. We will freely work with $(\infty,1)$-categories in the sense of \cite{HTT} and if we write $\infty$-category, we will always mean $(\infty,1)$-category. When we say \textit{category}, we always mean $1$-category and view it as an $\infty$-category via the nerve functor.
\begin{enumerate}
	\item[$\bullet$] Let $K$ be a poset. Then we view $K$ as an $\infty$-category with morphisms induced by the ordering. 
	\item[$\bullet$] If $F\colon C\rightarrow D$ is a functor of $\infty$-categories that admits a right adjoint $G$, then we will denote the adjunction of $F$ and $G$  by the symbol $F \dashv G$. 
	\item[$\bullet$] The $\infty$-category of small $\infty$-groupoids is denoted by $\SS$.
	\item[$\bullet$] The $\infty$-category of small $\infty$-categories is denotes by $\ICat$.
	\item[$\bullet$] The $\infty$-category of presentable $\infty$-categories with continuous functors, i.e. colimit preserving functors, is denoted by $\Pr^{L}$. 
	\item[$\bullet$] We denote by $\PrLO$ the $\infty$-operad, whose underlying $\infty$-category is $\Pr^{L}$.
	\item[$\bullet$] The $\infty$-category of presentable stable $\infty$-categories with continuous functors is denoted by $\PrStL\coloneqq \mathrm{Mod}_{\Sp}(\PrLO)$.
	\item[$\bullet$] Let $\Dcal(\QQ)$ denote the category derived category of $\QQ$-vector spaces. This is an algebra object in $\PrStL$. The $\infty$-category $\DGCat$ denotes the $\infty$-category of $\Dcal(\QQ)$-modules in $\PrStL$.
	\item[$\bullet$] Throughout this article, we will deal with diagrams $X\in \Fun(K,\Ccal)$, where $K$ and $\Ccal$ are $\infty$-categories. For any $k\in K$, we will denote by $X_{k}$ the value of $X$ at $k$ and for any morphism $\alpha\colon k\rightarrow l$ in $K$, we will denote the induced morphism in $\Ccal$ by $X_{\alpha}$. We will often write $X_{kl}$ whenever the morphism $k\rightarrow l$ is clear from the context. This abusive notation is for convenience of the reader. \par 
	If $f\colon X\rightarrow Y$ is a morphisms of diagrams in $\Fun(K,\Ccal)$, then we write $f_{k}\colon X_{k}\rightarrow Y_{k}$ for the morphism in $\Ccal$ associated to $k\in K$.
\end{enumerate}
\subsubsection*{Algebraic notation}
Throughout, let $S$ be an excellent noetherian scheme of finite dimension. 
\begin{enumerate}
	\item[$\bullet$] An algebraic stack is in the sense of \cite{stacks-project}, also referred to as Artin stack.
	\item[$\bullet$]  Let $X$ be an $S$-scheme of finite type. By $\DM(X)$, we denote the DG-category of Beilinson-motives \cite[\S 14]{CD}. Roughly, these are $\AA^{1}$-local \'etale sheaves on $X$ with values in $\Dcal(\QQ)$ such that $\PP^{1}$ is $\otimes$-invertible.
\end{enumerate}

\section{Motives on pro-algebraic stacks}

In this article, we aim to extend the six-functor formalism to limits of algebraic stacks, such as the stack of Barsotti–Tate groups. This stack can be realized as a limit of its truncations. We develop a general theory for inverse systems of algebraic stacks, allowing for non-representable transition maps.\par

Before introducing our main constructions and definitions, we recall the theory of rational motives on algebraic stacks. Following the construction of Richarz–Scholbach, we obtain the DG-category of rational motives on prestacks as defined in \cite[Def.~2.2.1]{RS1}. We will later specialize this to our specific context. In what follows, we omit some details and refer the reader to \textit{loc.~cit.} for an explicit construction.\par

Let $\AffSch^{\ft}_{S}$ denote the category of affine schemes of finite type over $S$. Since $\AffSch^{\ft}_{S}$ is essentially small, we replace it with a small skeleton consisting of the relevant objects for our purposes. We write $\AffSch^{\kappa}_{S}$ for the pro-$\kappa$-completion of $\AffSch^{\ft}_{S}$, which is again small by \cite[Lem.~6.1.2]{KaSch}.\par

A prestack over $S$ is a functor from $\AffSch^{\kappa}_{S}$ to $\infty$-groupoids. The $\infty$-category of prestacks over $S$ is denoted by $\PreStk_{S}$.\par

The functor $\DM \colon X \mapsto \DM(X)$ denotes the functor of Beilinson motives on finite type $S$-schemes. We extend $\DM$ to prestacks as follows. First, we perform a left Kan extension of $\DM$ from $\AffSch^{\ft}_{S}$ to $\AffSch^{\kappa}_{S}$, followed by a right Kan extension along the Yoneda embedding $\AffSch^{\kappa}_{S} \hookrightarrow \PreStk_{S}$. The transition maps in this process are given by $*$-pullbacks. This yields a functor
$$
	\DM^{*} \colon \PreStk_{S}^{\op} \to \DGCat.
$$
If we instead use $!$-pullbacks in place of $*$-pullbacks, we denote the resulting functor by $\DM^{!}$. However, when restricted to schemes locally of finite type over $S$, these two constructions are equivalent. This equivalence follows from the fact that, in our setting, $\DM$ satisfies h-descent on schemes \cite{CD}. Consequently, both $\DM^{!}$ and $\DM^{*}$ can be computed via descent. Indeed, schemes locally of finite type admit open covers by affine schemes of finite type, and $!$- and $*$-pullbacks agree up to shift and twist. The same reasoning shows that $\DM^{*}(X) \simeq \DM^{!}(X)$ for any algebraic space, since the diagonal of an algebraic space is representable by a scheme.\par

Using étale descent, we may glue the six-functor formalism of $\DM$ to the restriction of $\DM^{*}$ to algebraic spaces, as in \cite{KhanVFC}, or equivalently via the DESCENT algorithm of Liu–Zheng \cite{LiuZheng}. Similar constructions appear in \cite{RS1}, where ind-algebraic stacks and ind-schemes are treated.\par

We may now apply the same arguments to glue the six-functor formalism to the restriction of $\DM^{*}$ to algebraic stacks locally of finite type over $S$. As before, the existence of a smooth cover implies that $\DM^{!}(X) \simeq \DM^{*}(X)$ for any such stack $X$. 

To summarize, there is (up to equivalence) no distinction between working with $\DM^{!}$ and $\DM^{*}$. Furthermore, we obtain a complete six-functor formalism in the sense of Liu–Zheng for algebraic stacks locally of finite type over $S$, i.e., the functor $\DM^{*}$ extends to a symmetric monoidal functor
$$
	\DM^{*} \colon \Corr{\mathrm{Stk}^{\lft}_{S}}{\lft} \to \mathrm{CAlg}(\mathrm{Pr}^{L}_{\omega}).
$$
\par

To simplify notation, we will henceforth write $\DM$ for the extension $\DM^{*}$. We retain the notation $\DM^{!}$ for the $!$-variant.

 \begin{notation}
 	Throughout this article, we will denote by $\Ical$ a filtered $\infty$-category.
\end{notation}

\begin{defi}
\label{defi.pro-stack}
Recall the definition of a pro-$\Ical$-algebraic stack of Definition \ref{def.intro}. We simply say \textit{pro-algebraic} if we do not want to specify $\Ical$.
\begin{enumerate}
	\item[$\bullet$]	If $(x,X)$ is a pro-$\Ical$-algebraic stack, then we denote by $X_{i}\coloneqq x(i)$ and the transition maps $X_{j}\rightarrow X_{i}$ by $x_{ij}$, for $i\rightarrow j\in \Ical$.
%	\item[$\bullet$] We denote by $\ProAlg^{\Ical}$ the full subcategory of $\Fun(\Ical^{\op},\PreStk_{S})$ spanned by pro-$\Ical$-algebraic stacks.
	\item[$\bullet$] Let $f\colon (x,X)\rightarrow (y,Y)$ be a morphism of pro-$\Ical$-algebraic stacks. If $\Pbf$ is a property of a morphism of algebraic stacks, we will say that $f$ has property $\Pbf$ if each $f_{i}\colon X_{i}\rightarrow Y_{i}$ has property $\Pbf$. 
	\item[$\bullet$] If $X$ is an algebraic stack, we will abuse notation to denote the constant diagram $i\mapsto X$ (here the transition maps are given by $\id_{X}$) with $X$ itself.
\end{enumerate}
 \end{defi}
 
 We will define some important properties of pro-algebraic stacks inspired from examples of the introduction.
  
 \begin{defi}
 	Let $(x,X)$ be a pro-$\Ical$-algebraic stack. Then we say that $(x,X)$ is
	\begin{enumerate}
		 \item \textit{classical} if $x_{ij}$ is affine for all $i\rightarrow j\in\Ical$,
		\item \textit{strict} if $x_{ij}^{*}\colon \DM(X_{i})\rightarrow \DM(X_{j})$ is fully faithful for all $i\rightarrow j\in\Ical$.
		\item \textit{tame} if each $x_{ij}$ is smooth\footnote{In practice the pro-algebraic stacks we consider are always tame. However, not every assertion in our construction will need this condition.}.
	\end{enumerate}
 \end{defi}
 
 \begin{rem}
 	 The property strict will not be used heavily in this article. Nevertheless, we want to remark that the stack of displays by Lau admits a presentation as a strict pro-$\NN$-algebraic stack. 	  
	 Perhaps a more straightforward example is the infinite dimensional affine space, via the projections $\AA^{n+1}\rightarrow  \AA^{n}$.
 \end{rem}

As remarked in the introduction, we only get a $6$-functor formalism for certain classes of morphisms. Let us define these.
 
 \begin{defi}
 \label{def.adj}
 	If $f\colon(x,X)\rightarrow (y,Y)$ is a morphism of pro-$\Ical$-algebraic stacks, then we call $f$ \textit{cartesian} if each square 
$$
	\begin{tikzcd}
		X_{j}\arrow[r,"x_{ij}"]\arrow[d,"f_{j}",swap]& X_{i}\arrow[d,"f_{i}"]\\
		Y_{j}\arrow[r,"y_{ij}"]&Y_{i}
	\end{tikzcd}
$$
is a pullback diagram. We call a morphism $f\colon(x,X)\rightarrow (y,Y)$ of pro-$\Ical$-algebraic stacks $*$\textit{-adjointable} if the exchange morphism $f_{i}^{*}y_{ij*}\to x_{ij*}f_{j}^{*}$ induced by each square above is an equivalence and $!$\textit{-adjointable} if $f_{j!}x_{ij}^{!}\rightarrow y_{ij}^{!}f_{i!}$ is an equivalence.
 \end{defi}

 The property \textit{cartesian} is perhaps a more natural restriction for the existence of a $6$-functor formalism. Especially, when considering that homotopy invariance, T-stability and localization requires a good notion of the affine line, the projective space resp. taking complement in the category of pro-algebraic stacks. However, it turns out that the property \textit{adjointable} is enough for just the existence of a $6$-functor formalism if $\Ical\simeq \NN$. We will make this more precise late but let us give an example of adjointable morphisms of pro-algebraic stacks.

 \begin{example}
 \label{ex.adjointable.mor}
\begin{enumerate}
	\item[(1)] Let $f\colon (x,X)\rightarrow (y,Y)$ be a  cartesian morphism of pro-$\Ical$-algebraic stacks, then $f$ is adjointable.
	\item[(2)] Let $f\colon (x,X)\rightarrow (y,Y)$ be a morphism of pro-$\Ical$-algebraic stacks. If $y_{ij*}$ is fully faithful for all $i\to j\in \Ical$ (e.g. when $(y,Y)\cong S$) and $(x,X)$ is strict, then a simple computation shows that $f$ is $*$-adjointable.
	
	\hspace{8pt} Indeed, we use that $\id\to x_{ij*}x^{*}_{ij}$ and $y_{ij}^{*}y_{ij*}\to\id$ are invertible to conclude that the morphism  
	$$
		f_{i}^{*}y_{ij*}\rightarrow x_{ij*}x^{*}_{ij}f_{i}^{*}y_{ij*}\simeq x_{ij*}f_{j}^{*}y_{ij}^{*}y_{ij*}\rightarrow x_{ij*}f_{j}^{*}
	$$
	is also invertible.  
	
	\hspace{8pt} If $(x,X)$ and $(y,Y)$ in addition tame and $\id\to y_{ij}^{*}y_{ij\sharp}$ is invertible (again this holds if $(y,Y)\cong S$), then the same argument together with purity shows that $f$ is $!$-adjointable.
\end{enumerate}
\end{example}

\begin{example}
\label{ex.alg.ext}
	Let $L/K$ be an algebraic extension of fields. Let $\Fcal^{L}_{K}$ be the filtered category of all finite field extensions $L/E/K$. Let $\Fcal^{L}_{K}\rightarrow \textup{Fields}$ be the associated diagram in the category of fields. Then $\Spec(L)\cong \lim_{E\in \Fcal^{L}_{K}}\Spec(E)$. Now let $\Xfr$ be an algebraic stack over $K$. For any $E\in \Fcal^{L}_{K}$, we set $\Xfr_{E}\coloneqq \Xfr\times_{K} E$. Let $x_{E'/E}\colon \Xfr_{E}\rightarrow \Xfr_{E'}$ be the projection for a field extension $E'/E$. Then $(x,\Xfr_{L})$ is a classical pro-$\Fcal^{L}_{K}$-algebraic stack. Naturally, we also get a classical pro-$\Fcal^{L}_{K}$-algebraic stack $(\iota,\Spec(L))$ defined by the inclusions of field extensions in $\Fcal^{L}_{K}$. The associated morphism $(x,\Xfr_{L})\rightarrow (\iota,\Spec(L))$ is cartesian. Note that $(\iota,\Spec(L))$ and thus also $(x,\Xfr_{L})$ is not strict as for example any Galois extension $E\hookrightarrow E$ is a $\Gal(E'/E)$-torsor.\par 
	The above can be generalized to the following case. We can replace $\Spec(K)$ by $S$ and assume that $S$ is connected. Let $\sbar$ be a geometric point of $S$. Then we can define $F_{S}$ as the category of pointed \'etale covers over $(S,\sbar)$. The construction above can be repeated in this case for any algebraic stack $\Xfr$ over $S$.
\end{example}

Now let us define the DG-category of rational motives associated to a pro-algebraic stack.
 
 \begin{defi}
 \label{def.DM}
 	Let $(x,X)$ be a pro-$\Ical$-algebraic stack. Then we define the \textit{$\infty$-category of rational motives on $(x,X)$} as $$\DM(x,X)\coloneqq \colim_{i\in\Ical,x^{*}}\DM^{*}(X_{i}),$$ where  the colimit is taken in $\DGCat$.
 \end{defi}
 
 One might wonder why we chose $\DM^{*}$ and not $\DM^{!}$. Our main reason is convenience. We conjecture that we have an equivalence between both possible definitions. But by working with $\DM^{*}$, monoidality of $\DM(x,X)$ follows by definition. Also, we will see that if $\Ical\simeq \NN$, then both possible definitions are equivalent.
 
Let us now formalize Theorem \ref{thm.intro.2} of the introduction. We will postpone the technical proof to Section \ref{sec.proof.6-ff}.
 
\begin{thm}
\label{thm.6-ff}
	Let  $f\colon (x,X)\rightarrow (y,Y)$ be a morphism of pro-$\Ical$-algebraic stacks. Then there is an adjunction
	$$
		 \begin{tikzcd}
			f^{*}\colon \DM(y,Y)\arrow[r,"",shift left = 0.3em]&\arrow[l,"",shift left = 0.3em]\DM(x,X)\colon f_{*},
		\end{tikzcd}
	$$
	where $f^{*}$ is symmetric monoidal and preserves compact objects. The functor $f^{*}$ extends to a functor $(x,X)\mapsto \DM(x,X)$. Moreover, the functor $\DM$ satisfies the following property on pro-$\Ical$-algebraic stacks.
	\begin{enumerate}		
		\item[(PB)] If $f$ is smooth and $*$-adjointable, then $f^{*}$ admits a left adjoint $f_{\sharp}$ satisfying base change and projection formula.
	\end{enumerate}
	
	Assume now that $f$ locally of finite type. Further, assume one of the following holds true
	\begin{enumerate}
		\item[\textup{($\ast$)}] $f$ is cartesian, or
		\item[\textup{($\dagger$)}] $\Ical\simeq \NN$, $f$ is $!$-adjointable, $(x,X)$ and $(y,Y)$ are tame.
	\end{enumerate}
	
	 Then we can upgrade $f^{*}$ to be part of a $6$-functor formalism in the following sense. There exists an adjunction 
	$$
	 \begin{tikzcd}
			f_{!}\colon \DM(x,X)\arrow[r,"",shift left = 0.3em]&\arrow[l,"",shift left = 0.3em]\DM(y,Y)\colon f^{!},
		\end{tikzcd}
	$$
	where $f_{!}$ defines a covariant functor on pro-algebraic stacks satisfying $(*)$ or $(\dagger)$.
	Further, the adjunctions above satisfy the following properties.
	\begin{enumerate}
		\item[(F1)] If $f$ is cartesian and proper, we have a canonical equivalence $f_{*}\simeq f_{!}$.
		\item[(F2)] The functor $f^{!}$ commutes with arbitrary colimits.
		\item[(F3)] (Base change) For any pullback square
	$$
	\begin{tikzcd}
		(w,W)\arrow[r,"f'"]\arrow[d,"g'",swap]& (z,Z)\arrow[d,"g"]\\
		(x,X)\arrow[r,"f"]&(y,Y)
	\end{tikzcd}
	$$
	of pro-algebraic stacks with morphisms satisfying $(*)$ resp. $(\dagger)$, the exchange morphisms
	\begin{align*}
		f^{*}g_{!}\rightarrow g'_{!}f'^{*}\\
		g^{!}f_{*}\rightarrow f'_{*}g'^{!}
	\end{align*}
	are equivalences.
	\item[(F4)]  (Projection formula) The exchange transformation $f_{!}( - \otimes f^{*}(-))\rightarrow f_{!}(-)\otimes -$ is an equivalence.	
\end{enumerate}
Furthermore, the following properties hold.
\begin{enumerate}
		\item[(M1)] (Homotopy invariance) Let us define $(q,\AA^{1}_{X})$ via the system $i\mapsto \AA^{1}_{X_{i}}$. Then $p^{*}$ is fully faithful, where  $p\colon (q,\AA^{1}_{X})\rightarrow (x,X)$ denotes the projection.
		\item[(M2)] (Localization) Let $i\colon (z,Z)\hookrightarrow (x,X)$ be cartesian closed immersion of pro-algebraic stacks with open complement $j\colon (u,U)\rightarrow (x,X)$. Then there exist fiber sequences
		\begin{align*}
			i_{!}i^{!}\rightarrow \id \rightarrow j_{*}j^{*}\\
			j_{!}j^{!}\rightarrow \id \rightarrow i_{*}i^{*}.
		\end{align*}
			\end{enumerate}
	
	Assuming that $\Ical$ admits an initial object $0\in \Ical$, we moreover have the following properties.
	\begin{enumerate}

		\item[(M3)] (T-stability) Let $V_{0}\rightarrow X_{0}$ be a vector bundle and let $p\colon (v,V)\rightarrow(x,X)$ be the induced morphism of pro-algebraic stacks. The zero section $s_{0}\colon X_{0}\rightarrow V_{0}$ induces a section $s\colon (x,X)\rightarrow (v,V)$. Then the Thom-map $\textup{Th}(p,s)\coloneqq \textup{Th}(V_{0})\coloneqq p_{\sharp}s_{*}$ is an equivalence of DG-categories. 
		\item[(M4)] (Orientation) Let $(q,\Gm_{,X})$ be the pro-algebraic stack defined by $n\mapsto \Gm_{,X_{i}}$. We naturally obtain a projection $p\colon (q,\Gm_{,X})\rightarrow (x,X)$. For $M\in\DM(X)$, we set $M(1)\coloneqq \cofib(p_{\sharp}p^{*}M\rightarrow M)[-1]$. Let $V_{0}\rightarrow X_{0}$ be a vector bundle of rank $n$. Then there is a functorial equivalence $\textup{Th}(V_{0})1_{X}\simeq 1_{X}(n)[2n]$.
		\item[(M5)] (Purity) If $f$ is cartesian, we have a canonical equivalence $f^{!}\simeq \Th(\Omega_{f_{0}})1_{X}\otimes f^{*}$.
	\end{enumerate}
	
If we further restrict ourselves to pro-$\Ical$-schemes, i.e. inverse systems of schemes locally of finite type over $S$ indexed over $\Ical$, then $f^{*}$ extends to a $6$-functor formalism in the sense of Liu-Zheng \cite{LiuZheng}.

\end{thm}

Before we continue with the proof, we want to show that $\DM(x,X)\simeq \DM(X)$ for classical pro-algebraic stacks, as one would expect. We will need the following result.

\begin{lem}
\label{lem.Gai.colim}
	Let $\Ecal$ be an $\infty$-category that admits finite limits. Let $F\colon \Ecal\rightarrow \Pr^{L}_{\omega}$ be a functor. Denote by $F_{\Ical}$ the composition of $\Fun(\Ical,F)$ with the colimit functor $\Fun(\Ical,\Pr^{L}_{\omega})\rightarrow \Pr^{L}_{\omega}$. Then we have the following.
	\begin{enumerate}
		\item[(1)] Let $X\in \Fun(\Ical,\Ecal)$. We denote by $X^{\op}$ the opposite functor. Further, let us denote by $F_{*}\colon \Ecal^{\op}\rightarrow \Pr^{R}$ the functor induced by the equivalence $\Pr^{L}\simeq (\Pr^{R})^{\op}$ \cite[Cor. 5.5.3.4]{HTT}. Then we have
	$$
		F_{\Ical}(X) \simeq \lim_{\Pr^{R}} F_{*} \circ X^{\op}, 
	$$
	where the equivalence is in $\ICat$. 
		\item[(2)]  For any $i_{0},i_{1}\in \Ical$ the composition
	$$
		F(X_{i_{0}})\xrightarrow{\ins_{i_{0}}^{X}} F_{\Ical}(X)\simeq \lim_{\Pr^{R}} F_{*} \circ X\xrightarrow{p_{i_{1}}^{X}} F(X_{i_{1}})
	$$
	is canonically equivalent to 
	$$
		\colim_{k\in \Ical, \alpha\colon i_{0}\rightarrow k, \beta\colon i_{1}\rightarrow k} F_{*}(X_{\beta})\circ F_{\Ical}(X_{\alpha})
	$$
	where the colimit is taken in $\Fun_{\Pr^{L}}(F(X_{i_{0}}), F(X_{i_{1}}))$.
	\end{enumerate}
\end{lem}
\begin{proof}
	The first assertion follows from \cite[Cor. 5.5.3.4]{HTT} (see also Remark \ref{rem.Gai.colim} below). A proof for (2) is given by Gaitsgory in the setting of DG-categories (cf. \cite[Lem. 1.3.6]{Gai} - note that the proof is completely analogous in this slightly more general setting).
\end{proof}

\begin{rem}
\label{rem.Gai.colim}
	The equivalence in Lemma \ref{lem.Gai.colim} for $X\in \Fun(\Ical,\Ecal)$ is induced via the following process \cite[Lem. 1.3.2]{Gai}. The projections $p^{X}_{i}\colon\lim_{\Pr^{R}} F_{*} \circ X\rightarrow X_{i}$ admit left adjoints, denoted by $\gamma_{i}^{X}$. These left adjoints assemble to an equivalence
	\begin{equation}
\label{eq.pres.colim}
		h_{X}\colon F_{\Ical}(X)\xrightarrow{\sim}  \lim_{i\in \Ical^{\op},x_{*}}F(X_{i}).
	\end{equation}
	This is an equivalence in $\Pr^{L}$, where we view the RHS naturally as an object in $\Pr^{L}$.
	The inverse of $h_{X}$ is equivalent to $\colim_{i\in \Ical} \ins^{X}_{i}\circ p_{i}^{X},$
	where $\ins^{X}_{i}\colon F(X_{i})\rightarrow F_{\Ical}(X)$ denotes the canonical map.
\end{rem}

Let us fix the above notation, as we will need it later on.

\begin{notation}
	We fix the notation as in Remark \ref{rem.Gai.colim}, i.e. for a pro-$\Ical$-algebraic stack $(x,X)$, we denote by $h_{X}$ the equivalence 
$$
	h_{X}\colon \DM(x,X)\rightarrow \lim_{i\in \Ical,x_{*}} \DM(X_{i}),
$$
we denote for any $i\in \Ical$ by 
$$
	\ins_{i}^{X}\colon \DM(X_{i})\rightarrow \DM(x,X) \textup{ resp. } p^{X}_{i}\colon \lim_{i\in \Ical,x_{*}} \DM(X_{i})\rightarrow \DM(X_{i})
$$
the canonical inclusion resp. projections and we set $\gamma_{i}^{X}\coloneqq h_{X}\circ\ins_{i}^{X}$.
\end{notation}

 \begin{prop}
\label{prop.pres.colim}
	Assume that $\Ical$ admits an initial object $0\in\Ical$. Let $(x,X)$ be a classical pro-$\Ical$-algebraic stack. Then the natural map
	$$
		 \DM(x,X)\rightarrow \DM(X)
	$$
	is an equivalence of DG-categories.
\end{prop}
\begin{proof}
For the proof we will use two immediate facts. The canonical map $\DM(x, X)\rightarrow
\DM(X)$ is by construction monoidal and preserves colimits. Moreover, to show that $\DM(x, X)\rightarrow
\DM(X)$ is an equivalence, we may work with the right adjoints
$$ 
	\DM(X)\rightarrow \lim_{i\in \Ical^{\op},x_{*}} \DM(X_{i})
$$
(this follows from Lemma \ref{lem.Gai.colim} applied to $F = \DM: \Stk^{\lft,\op}_{S} \rightarrow \textup{Pr}^{L}_{\omega}$).
By assumption $X_{0}$ admits a smooth cover by affine finite type $S$-schemes $U_{0,k}= \Spec(A_{k})$, for some set $K$ and $k\in K$. As $•X_{i}\rightarrow X_{0}$ is affine for all $i\in \Ical$, the pullback $U_{i,k}\coloneqq X_{i}\times_{X_{0}}U_{i,k}$ is also a smooth cover by affines. The same holds after passage to limits, i.e. for $U_{k}\coloneqq X\times_{X_{0}}U_{0,i}$ the projection $\coprod_{k\in K}U_{i}\rightarrow X$ is an effective epimorphism \cite[01YZ]{stacks-project}. In particular, as $\DM$ satisfies $h$-descent (cf. \cite[Thm. 2.2.16]{RS1}), we have $\lim_{\Delta}\DM(\Cv_{\bullet}(\coprod_{k\in K}U_{i}/X))\simeq \DM(X)$. By construction, each of the $U_{i}$ is equivalent to $\lim_{i\in\Ical} U_{i,k}$. As the $U_{i,k}$ are affine of finite type over $S$, we have 
$$
	\DM(U_{i})\simeq \DM(\lim_{i\in\Ical} U_{i,k})\simeq \colim_{i\in\Ical, x^{*}}\DM(U_{i,k})\simeq \lim_{i\in\Ical^{\op},x_{*}}\DM(U_{i,k}).
$$
This yields the equivalence
$$
	\DM(\Cv_{\bullet}(\coprod_{k\in K}U_{i}/X))\simeq  \lim_{i\in\Ical,x_{*}}\DM(\Cv_{\bullet}(\coprod_{k\in K}U_{i,k}X_{i})).
$$
Putting all of this together,  we claim that this yields
\begin{align*}
	\DM(X)\simeq \lim_{\Delta}\DM(\Cv_{\bullet}(\coprod_{k\in K}U_{i}/X)) &\simeq \lim_{i\in\Ical^{\op},x_{*}} \lim_{\Delta}\DM(\Cv_{\bullet}(\coprod_{k\in K}U_{i,k}/X_{i}))\\ &\simeq  \lim_{i\in\Ical^{\op},x_{*}} \DM(X_{i})\simeq \colim_{i\in\Ical, x^{*}} \DM(X_{i})
\end{align*}
as desired.\par 
Indeed, the only thing to check is the second equivalence. For this it is enough to see that $x_{ij}^{*}$ induces a map of simplicial objects 
$$
	\DM(\Cv_{\bullet}(\coprod_{k\in K}U_{j,k}/X_{j}))\rightarrow \DM(\Cv_{\bullet}(\coprod_{k\in K}U_{i,k}/X_{i})).
$$
But this follows immediately from smooth base change, as 
$$
	\Cv_{\bullet}(\coprod_{k\in K}U_{j,k}/X_{j})\simeq \Cv_{\bullet}(\coprod_{k\in K}U_{i,k}/X_{i})\times_{X_{i}}X_{j}.
$$
\end{proof}

\begin{rem}
	Let $Y$ be a scheme and let $(x, X)$ be a classical pro-algebraic stack such that each $X_{i}$ is representable by a finite type $Y$-scheme. The proof of Proposition \ref{prop.pres.colim} can be repeated to obtain the following more general result on schemes.
\begin{enumerate}
	\item[$\bullet$] Let $F\colon \Sch^{\ft,\op}_{Y}\rightarrow \textup{CALg}(\PrLO)$ be a Zariski sheaf, then $\colim F \circ x \simeq F(X)$.
\end{enumerate}
 This includes for example motives with integral coefficients and the stable homotopy category of Morel-Voevodsky.
\end{rem}

By definition if $(x,X)$ is a strict pro-algebraic stack then the transition maps $x_{ij}^{*}$ are fully faithful for all $i\rightarrow j\in \Ical$. In particular, we would expect that the pullback $\DM(X_{0})\rightarrow \DM(X)$ is fully faithful. At least when we work with $\DM(x,X)$ this is true, as our explicit analysis will show. If $(x,X)$ is classical, we see in particular that $\DM(X_{0})\rightarrow \DM(X)$ is indeed fully faithful.

\begin{rem}[Underling motive of strict pro-algebraic stacks]
\label{rem.underlying.motive}
Let us show how to compute the underlying motive in $\DM(S)$  of a strict pro-$\Ical$-algebraic stack, when $\Ical$ admits an initial object $0$.\par
	Let $(x,X)$ be a strict pro-algebraic stack. Then by Example \ref{ex.adjointable.mor} (2), the projection $c_{0}\colon (x,X)\rightarrow X_{0}$ is an adjointable map of pro-algebraic stacks. Thus, by Theorem \ref{thm.6-ff}, we obtain an adjunction
	$$
		 \begin{tikzcd}
			c_{0*}\colon \DM(x,X)\arrow[r,"",shift left = 0.3em]&\arrow[l,"",shift left = 0.3em]\DM(X_{0})\colon c_{0}^{*}
		\end{tikzcd}
	$$
	Moreover, we see from explicit computations in Lemma \ref{lem.Gai.colim} (2) that $c_{0}^{*}$ is fully faithful, since $c_{0}^{*}$ corresponds to canonical map $\DM(X_{0})\rightarrow \DM(x,X)$ and $c_{0*}$ to the projection.
	\par 
	In particular, we will see ath if $f\colon (x,X)\rightarrow S$ is strict, tame and $f_{0}\colon X_{0}\rightarrow S$ denote the structure map, we have for any $M\in \DM(S)$ the equivalence $f_{!}f^{!}M\simeq f_{0!}f_{0}^{!}M$.
\end{rem}

\section{(Motivic) six functor formalism on pro-algebraic stacks}

\label{sec.proof.6-ff}

In this section we will prove Theorem \ref{thm.6-ff}.

Moreover, we will not restrict us to the category of algebraic stacks but rather work in the most general context, as we will only need formal properties, such as a geometric setup in the sense of Mann \cite{Mann}. While this is a more abstract formalism we develop, the benefit is that we can also apply this formalism to other settings, for example ind-schemes. Such inverse limits of ind-schemes and \'etale cohomology were considered by Bouthier-Kazhdan-Varshavsky \cite{Kaz} where they were called \textit{placid}.

\subsection{Pullback formalisms}
\label{sec.pull}
We start by analyzing filtered colimits of pullback formalisms after Drew-Gallauer \cite{DG}. We will analyze such structures more generally in this subsection without using pro-algebraic stacks but work with general diagrams in an $\infty$-category $\Ecal$ admitting finite limits.

\begin{defi}[\cite{DG}]
\label{def.pf}
Let $\Ecal$ be an $\infty$-category that admits finite limits and let us fix a class of morphisms $P\subseteq \Fun(\Delta^{1},\Ecal)$ that is closed under pullback, equivalences and composition.
	A \textit{compactly generated pullback formalism on $\Ecal$ with respect to $P$} is a functor 
	$$
		\Mcal^{\otimes}\colon \Ecal^{\op}\rightarrow \textup{CAlg}(\PrLO_{\omega}),\quad f\mapsto f^{*}\quad\textup{(we denote the right adjoint of $f^{*}$ by $f_{*}$)}
	$$
	satisfying the following conditions.
	\begin{enumerate}
		\item[(1)] If a morphism $f\colon X\rightarrow Y$ of $\Ecal$ lies in $P$, then there exists a left adjoint $f_{\sharp}$ of $f^{*}$.
		\item[(2)] For each pullback square
		$$
			\begin{tikzcd}
				X\times_{Z} Y\arrow[r,"g'"]\arrow[d,"f'"]& X\arrow[d,"f"]\\
				Y\arrow[r,"g"]&Z
			\end{tikzcd}
		$$
		in $\Ecal$ such that $f$ lies in $P$, the exchange transformation
		$$
			f'_{\sharp}g'^{*}\rightarrow f'_{\sharp}g'^{*}f^{*}f_{\sharp}\rightarrow f'_{\sharp}f'^{*}g^{*}f_{\sharp}\rightarrow g^{*}f_{\sharp}
		$$
		is equivalence.
		\item[(3)] For each  morphism $f\colon X\rightarrow Y$ in $P$ the exchange transformation
		$$
			f_{\sharp}(M\otimes_{\Mcal^{\otimes}(X)} f^{*}N)\rightarrow f_{\sharp}M\otimes_{\Mcal^{\otimes}(Y)} N
		$$
		is an equivalence, for any $M\in \Mcal^{\otimes}(X)$ and $N\in \Mcal^{\otimes}(Y)$.
	\end{enumerate}
	We call the triple $(\Mcal^{\otimes},\Ecal,\Pcal)$ as above a \textit{pullback formalism}. 
\end{defi}

\begin{rem}
\label{rem.exchange.defi}
	For property (2) in Definition \ref{def.pf} we may equivalently ask for the $*$-exchange morphism $f^{*}g_{*}\rightarrow g'_{*}f'^{*}$ to be an equivalence \cite[Prop. 1.1.9]{ayoub}.
\end{rem}

\begin{example}
\label{ex.pf.1}
	Recall that in the assumptions \ref{notations}, we fixed an excellent noetherian scheme $S$ finite dimension. By \cite[Thm. 16.1.4, Cor. 6.2.2]{CD} the $\infty$-category $\DM(X,\QQ)$ is compactly generated for any finite type $S$-scheme and $f^{*}$ preserves compact objects. In particular, $\DM_{\QQ}$ induces a pullback formalism on finite type $S$-schemes. By descent this can be extended to locally of finite type algebraic stacks over $S$ \cite{RS1}.
\end{example}
 
We will start by extending $\Mcal^{\otimes}$ to diagrams in $\Ical$, via taking colimits.

\begin{rem}
\label{rem.coef.pres}
	Let us remark that $\Pr^{L}_{\omega}$ is presentable and thus so is $\textup{Mod}_{\Sp}(\PrLO_{\omega})$ and $\textup{CAlg}(\PrLO_{\omega})$ \cite[Lem. 5.3.2.9, Cor. 4.2.3.7, Prop. 3.2.3.5]{HA}.
\end{rem}
 
\begin{notation}
	Consider the functor
$$\Fun(\Ical,\Ecal^{\op})\rightarrow \Fun(\Ical,\textup{CAlg}(\PrLO_{\omega}))
$$
induced by postcomposition with $\Mcal^{\otimes}$.
Composing this functor with the colimit functor, we obtain
$$
	\Mcal^{\otimes}_{\Ical}\colon \Fun(\Ical,\Ecal^{\op})\rightarrow \textup{CAlg}(\PrLO_{\omega}).
$$
\end{notation}

If $(x,X)$ is a pro-algebraic stack, we see that the following notation agrees Definition \ref{def.DM}. 
 
 \begin{notation}
 	Let $\Mcal^{\otimes}$ be a pullback formalism on locally of finite type Artin $S$-stacks with respect to smooth morphisms. Let $(x,X)$ be a pro-$\Ical$-algebraic stack. Then we define $$\Mcal(x,X)\coloneqq \Mcal_{\Ical}^{\otimes}(x).$$
 \end{notation}

 \begin{prop}
 \label{prop.ind.pf}
 	Let $(\Mcal^{\otimes},\Ecal,\Pcal)$ be a pullback formalism. Let $\Pcal_{\Ical}$ denote the class of morphisms in $\Fun(\Ical,\Ecal^{\op})$  such that $f\in \Fun(\Ical,\Ecal^{\op})$ belongs to $\Pcal_{\Ical}$ if and only if each $f_{i}$ belongs to $\Pcal$ and each square in $\Mcal_{\Ical}(f)$ is adjointable, i.e. for any morphism $i\rightarrow j\in \Ical$ the induced morphism $f_{i}^{*}y_{ij*}\rightarrow x_{ij*}f_{j}^{*}$ is an equivalence \cite[Def. 4.7.4.13]{HA}. Then $\Mcal_{\Ical}$ is a pullback formalism on $\Fun(\Ical,\Ecal^{\op})$ for the class $\Pcal_{\Ical}$.
 \end{prop}
 
 Proposition \ref{prop.ind.pf} is one of the crucial points in extending the motivic six functor formalism on $\DM$ for Artin-stacks to pro-algebraic stacks.
 But before we can prove this proposition, we need some additional lemmas.
  
 \begin{lem}
\label{lem.compact.gen}
	 Let $F\colon \Ecal^{\op}\rightarrow \Pr^{L}_{\omega}$ be a functor. Denote by $F_{\Ical}$ the composition of $\Fun(\Ical,F)$ with the colimit functor. Let $X$ be in $\Fun(\Ical,\Ecal^{\op})$. Then we have 
	$$
		F_{\Ical}(X) \simeq \Ind(\colim_{i\in \Ical,x^{*}}F(X_{i})^{c})
	$$ 
	and any $M\in F_{\Ical}(X)$ is compact if and only if there exists an $i\in \Ical$ and $M_{i}\in F(X_{i})^{c}$ that lifts $M$.
\end{lem}
\begin{proof} 
	The first assertion follows from \cite[Cor. 7.2.7]{GR}, where the colimit on the right is taken in $\ICat$.\par 
    	Note that the filtered colimit of idempotent complete $\infty$-categories is idempotent complete \cite[Cor. 4.4.5.21]{HTT}. Thus, any compact object comes from an object of the colimit on the right hand side (cf. \cite[Lem. 5.4.2.4]{HTT} and \cite[02LG]{kerodon}).
\end{proof}

\begin{lem}
\label{lem.Gai.colim.2} If $f\colon X\rightarrow Y$ is a morphism in $\Fun(\Ical,\Ecal^{\op})$ such that each $f_{i}$ admits a left adjoint for all $i\in \Ical$ and each square is adjointable (in the sense of Proposition \ref{prop.ind.pf}), then $f^{*}\coloneqq F_{\Ical}(f)$ admits a left adjoint.
\end{lem}
\begin{proof}
This lemma is stated without proof in \cite[Prop. 5.1.8 (c)]{Kaz}. As far as we know, there is no written proof for this result in the literature, so we provide one.\par 
	By Lemma \ref{lem.Gai.colim} (1), we see that $f_{*}\colon  \lim_{\Ical} F_{*} \circ X \rightarrow \lim_{\Ical} F_{*} \circ Y$ admits a left adjoint $f^{*}$, that has to be compatible with the identification of $\lim_{\Ical} F_{*}$ and $F_{\Ical}$. We claim that the following diagram is commutative (up to homotopy)
	$$
		\begin{tikzcd}
			\lim_{\Ical} F_{*} \circ Y\arrow[r,"p_{i}^{Y}"]\arrow[d,"f^{*}"]& F(Y_{i})\arrow[d,"f_{i}^{*}"]\\
			\lim_{\Ical} F_{*} \circ X\arrow[r,"p_{i}^{X}"]&F(X)
		\end{tikzcd}
	$$
	for all $i\in \Ical$. This claim implies that $f^{*}$ commutes with limits and therefore admits a left adjoint $f_{\sharp}$.\par 
	Let us show the claim. For this, note that we have the following diagram with commutative squares (up to homotopy)
	$$
		\begin{tikzcd}
			F(Y_{j})\arrow[r,"\ins_{Y}^{j}"]\arrow[d,"f_{j}^{*}"]& F_{\Ical}(Y)\arrow[d,"f^{*}"]\arrow[r,"h_{Y}"]&\lim_{\Ical} F_{*} \circ Y\arrow[d,"f^{*}"]\\
			F(X_{j})\arrow[r,"\ins_{X}^{j}"]&F_{\Ical}(X)\arrow[r,"h_{X}"]&\lim_{\Ical} F_{*} \circ X 
		\end{tikzcd}
	$$
	for any $j\in \Ical$. First, we show that $p_{i}^{Y}\circ f^{*}$ commutes with colimits. Indeed, it is enough to see that $p_{i}^{Y}\circ f^{*}\circ h_{X}\circ \ins_{X}^{j}$ commutes with colimits, then is equivalent to the induced map $F_{\Ical}(X)\rightarrow F(Y_{i})$ in $\Pr^{L}$, which necessarily commutes with colimits. By Lemma \ref{lem.Gai.colim} (2), we have
$$
	p_{i}^{X}\circ f^{*}\circ h_{Y}\circ \ins_{Y}^{j}\simeq p_{i}^{Y}\circ h_{X}\circ \ins_{X}^{j}\circ f_{j}^{*}\simeq \colim_{i\rightarrow k, j\rightarrow k} X_{ik*}\circ X_{jk}^{*}\circ f_{j}^{*} 
$$
(recall our notations in the introduction under Section \ref{notations}).
By construction each of the functors $X_{ik*}, X_{jk}^{*}, f_{j}^{*}$ commutes with colimits and thus also the left hand side of the equivalence. Analogously, one can show that $f_{i}^{*}\circ p_{i}^{Y}$ commutes with colimits. Therefore, it is enough to show that 
$$
	p_{i}^{X}\circ f^{*}\circ h_{Y}\circ \ins_{Y}^{j}\simeq f_{i}^{*}\circ p_{i}^{Y}\circ h_{Y}\circ \ins_{Y}^{j}.
$$
The left hand side is equivalent to $\colim_{i\rightarrow k, j\rightarrow k} X_{ik*}\circ X_{jk}^{*}\circ f_{j}^{*}$ by the above and analogously, we see that the right hand side is equivalent to $\colim_{i\rightarrow k, j\rightarrow k} f_{i}^{*}\circ Y_{ik*}\circ Y_{jk}^{*}$. Since each square in $F(f)$ is right adjointable these colimits are in fact equivalent, proving the claim and therefore assertion.
\end{proof}

\begin{lem}
\label{lem.fsharp}
Let $f$ be a morphism in $\Fun(\Ical,\Ecal^{\op})$ such that $f_{i}$ admits a left adjoint for all $i\in \Ical$. Let $f_{\sharp}$ denote the left adjoint of $f$ given by Lemma \ref{lem.Gai.colim.2}. Then the following square 
	$$
		\begin{tikzcd}
			F(X_{i})\arrow[r,"\ins_{X}^{i}"]\arrow[d,"f_{i\sharp}"]& F_{\Ical}(X)\arrow[d,"f_{\sharp}"]\\
			F(Y_{i})\arrow[r,"\ins_{Y}^{i}"]&F_{\Ical}(Y)
		\end{tikzcd}
	$$
	commutes (up to homotopy) for all $i\in \Ical$.
\end{lem}
\begin{proof}
	By construction of $f_{\sharp}$ the square 
	$$
		\begin{tikzcd}
			F_{\Ical}(X)\arrow[r,"h_{X}"]\arrow[d,"f_{\sharp}"]& \lim_{\Ical}F_{*}\circ X\arrow[d,"f_{\sharp}"]\\
			F_{\Ical}(Y)\arrow[r,"h_{Y}"]&\lim_{\Ical}F_{*}\circ Y
		\end{tikzcd}
	$$
	commutes (up to homotopy). Thus, we may check if the square in the lemma commutes after applying $h_{X}$. But the proof of Lemma \ref{lem.Gai.colim.2} shows that the square 
	$$
		\begin{tikzcd}
			\lim_{\Ical} F_{*} \circ Y\arrow[r,"p_{i}^{Y}"]\arrow[d,"f^{*}"]& F(Y_{i})\arrow[d,"f_{i}^{*}"]\\
			\lim_{\Ical} F_{*} \circ X\arrow[r,"p_{i}^{X}"]&F(X_{i})
		\end{tikzcd}
	$$
	commutes (up to homotopy) for all $i\in \Ical$. So in particular, we conclude by passing to left adjoints.
\end{proof}

\begin{proof}[Proof of Proposition \ref{prop.ind.pf}]
	Assertion (1) follows from Lemma \ref{lem.Gai.colim.2}.\par
We are left to show (2) and (3). Invoking Lemma \ref{lem.Gai.colim} resp. Lemma \ref{lem.fsharp} and the fact that each compact of $\Mcal_{\Ical}(X)$ comes from some level $i\in \Ical$ (cf. Lemma \ref{lem.compact.gen}). So, we can check (2) and (3) on some $k\in\Ical$ (note that we use that $\Ical$ is filtered) and this follows from the definition of $\Mcal^{\otimes}$.
\end{proof}

\subsection{Six functor formalism}

In this section, we want to prove Theorem \ref{thm.6-ff}. We will first give a proof in the $(\dagger)$ case. Afterwards, we will give a proof in the $(\ast)$ case and complete the proof. 

Before we continue, we want to show that properties (M1)-(M5) are a formal consequence of the compact generation of $\DM(x,X)$ (see Lemma \ref{lem.compact.gen}).

\begin{lem}
\label{lem.M}
	In the setting of Theorem \ref{thm.6-ff}, assuming $f_{!}\dashv f^{!}$ exists, satisfying (F1)-(F4), the properties (M1) and (M2) hold. If $\Ical$ admits an initial object $0\in\Ical$, then also (M3) - (M5) hold.
\end{lem}
\begin{proof}
	Arguing as in the proof of Proposition \ref{prop.ind.pf} and using descent the properties (M1) - (M3) reduce to questions of $\DM$ on schemes, which is well known \cite{CD}. The property (M4) follows from the localization sequence.
	
	If $f$ is cartesian, smooth and then the relative dimension of $f_{i}$ and $f_{0}$ agree for all $i\in \Ical$. From here (M5) follows similarly as before.
\end{proof}

\subsubsection{Proof of Theorem \ref{thm.6-ff} for adjointable maps}

In this subsection, we will assume that we are in the setting $(\dagger)$ and we simply say pro-algebraic instead of pro-$\NN_{0}$-algebraic.

\begin{lem}
\label{lem.spine}
	Let $\Ccal$ be an $\infty$-category. Let $\textup{Spine}[\NN_{0}]$ be the simplicial subset of $\NN_{0}$ of finite vertices that are joined by edges (or more informally, we forget about all $n\geq 2$ simplices in $\NN_{0}$). Then the restriction $\Fun(\NN_{0},\Ccal)\rightarrow \Fun(\textup{Spine}[\NN_{0}],\Ccal)$ is a trivial Kan fibration.
	
	In particular, any natural transformation of functors $\alpha\colon \textup{Spine}[\NN_{0}]\times \Delta^{1}\rightarrow \Ccal$ can be lifted to a natural transformation in $\aghat\colon \NN_{0}\times \Delta^{1}\rightarrow \Ccal$.
\end{lem}
\begin{proof}
	Indeed, the functor $\textup{Spine}[\NN_{0}]\rightarrow \NN_{0}$ is inner anodyne \cite[03HK]{kerodon} proving the claim.
\end{proof}

\begin{prop}
\label{prop.conj.1}
	Let $(x,X)$ be a tame pro-algebraic stack, then 
	$$
		\DM(x,X)\simeq \colim_{i\in\NN_{0},x^{!}}\DM^{!}(X_{i}).
	$$
\end{prop}
\begin{proof}
		By descent, we see that for each $n\in \NN_{0}$ we have $\ptilde_{i}\colon \DM^{!}(X_{i})\simeq \DM^{*}(X_{i})$  as $X_{i}$ admits a smooth covering by a scheme. We will construct equivalences 
	$$
		p_{i}\colon \DM^{!}(X_{i})\simeq \DM^{*}(X_{i})
	$$
	out of $\ptilde_{n}$ that are compatible with the presentation as a colimit. This is easily achieved by twisting with the Thom-motives associated to the smooth maps $x_{0i}$. But let us be more precise.\par
		Let us define $p_{i}\coloneqq \ptilde_{n}\otimes \textup{Th}(\Omega_{x_{0n}})1_{X_{n}}$, where $\textup{Th}$ denotes the Thom-motive\footnote{Let $p\colon V\rightarrow X_{i}$ be a vector bundle with zero section  $s$, then $\textup{Th}(V)\coloneqq p_{\sharp}s_{*}$.} associated to $\VV(\Omega_{x_{0i}})\rightarrow X_{n}$. Then $p_{n}$ is still an equivalence of $\DM^{!}(X_{n})$ and $\DM^{*}(X_{n})$.  By compatibility with the purity equivalence with composition \cite[Rem. 2.4.52]{CD}, the $p_{i}$ are compatible with the presentation of $\DM(x,X)$ as a colimit by Lemma \ref{lem.spine}.
\end{proof}

\begin{rem}
	To generalize Proposition \ref{prop.conj.1} to arbitrary $\Ical$ it is enough to know that the level-wise equivalences $p_{i}$ in the proof yield an equivalence of diagrams. While on the homotopy categorical level this is clear, we were not able to write a map down for higher simplices.
\end{rem}

\begin{proof}[Proof of Theorem \ref{thm.6-ff}]
	The property (PB) is Proposition \ref{prop.ind.pf}. By Proposition \ref{prop.conj.1}, we obtain a functor $f^{!}\colon \DM(y,Y)\rightarrow \DM(x,X)$. Applying Proposition \ref{prop.ind.pf} to $f^{!}$, we obtain a left adjoint $f_{!}$ of  $f^{!}$. Note that we use that each $f_{n}^{!}$ preserves colimits, since $f_{n!}$ preserves compact objects. In particular, (F2) follows from construction. The properties (F3) and (F4) can be checked on compacts, which is clear by design and Lemma \ref{lem.compact.gen}. 
	
%	For the property (F1), we note that if $f$ is proper, then $f_{!}$ admits a left adjoint $\ftilde^{*}$. By Lemma \ref{lem.fsharp}, and Lemma \ref{lem.spine}, we get a map $\colim f_{\bullet}^{*}\rightarrow \ftilde^{*}$. By Lemma \ref{lem.compact.gen}, we see that this morphism is an equivalence. By construction, we obtain $f^{*}\simeq \colim f_{\bullet}^{*}\simeq  \ftilde^{*}$ and in particular $f_{!}\simeq f_{*}$ by uniqueness of adjoints.
	
	The rest follows from Lemma \ref{lem.M}
\end{proof}

\begin{rem}
Let $(x,X)$ be a strict tame pro-algebraic stack and $(y,Y)$ a tame pro-algebraic stack together with a morphism $f\colon (x,X)\to (y,Y)$. Assume that for $y_{n}\colon Y_{n+1}\to Y_{n}$ and $n\geq 0$ the unit $\id\to  y_{n}^{*}y_{n\sharp}$ is an equivalence. In this setting, we can make $f_{!}$ more explicit. 

	For each $n\geq 0$, we obtain a functor $f_{n!}$. We want to construct a morphism of diagrams $f_{\bullet!}\colon \NN_{0}\times \Delta^{1,\op}\rightarrow \DGCat$, which after taking the colimit then yields $f_{!}$. Since our morphisms are not cartesian, we have to take the Thom-twist of $x_{n}\colon X_{n+1}\to X_{n}$ and $y_{n}$ into account. 
	Set $\ftilde_{n!}\coloneqq (\mathrm{Th}(\Omega_{y_{0n}})1_{Y_{n}})^{-1}\otimes f_{n!}\circ \mathrm{Th}(\Omega_{x_{0n}})1_{X_{n}}\otimes -$. Then, we obtain 
	\begin{align*}
		y_{n}^{*}\ftilde_{n!}&\simeq y_{n}^{*}(\mathrm{Th}(\Omega_{y_{0n}})1_{Y_{n}})^{-1}\otimes f_{n!}\circ \mathrm{Th}(\Omega_{x_{0n}})1_{X_{n}}\otimes x_{n!}x_{n}^{!}\\ 
		&\simeq y_{n}^{*}(\mathrm{Th}(\Omega_{y_{0n}})1_{Y_{n}})^{-1}\otimes y_{n!}f_{n+1!}\circ \mathrm{Th}(\Omega_{x_{0n+1}})1_{X_{n+1}}\otimes x_{n}^{*}\\
		&\simeq y_{n}^{*}y_{n\sharp}\circ (\mathrm{Th}(\Omega_{y_{0n+1}})1_{Y_{n+1}})^{-1}\otimes f_{n+1!}\circ \mathrm{Th}(\Omega_{x_{0n+1}})1_{X_{n+1}}\otimes x_{n}^{*}\simeq \ftilde_{n+1!} x_{n}^{*},
	\end{align*}
	where we use compatibility of Thom twists with composition in the sense of \cite[Rem. 2.4.52]{CD} and the projection formula.
	
		In particular, we obtain the desired diagram $f_{\bullet!}\colon \NN_{0}\times \Delta^{1,\op}\rightarrow \DGCat$ that after taking colimits induces a functor $f_{!}\colon \DM(x,X)\rightarrow \DM(y,Y)$. This is a colimit preserving functor and therefore admits a right adjoint. By construction of the comparison morphism in Proposition \ref{prop.conj.1}, this agrees with the construction of $f_{!}$ in the proof of Theorem \ref{thm.6-ff}.
		
		This construction also shows why (F1) fails in this setting. Even when $f$ is proper, the resulting $f_{!}$ is an infinite twist of $f_{*}$. Note however, if $f$ is cartesian, then by design the twists appearing in $\ftilde$ cancel each other out, so this caveat appears only in the non-cartesian setting.
\end{rem}

\subsubsection{Proof of Theorem \ref{thm.6-ff} for cartesian maps}
We now proceed to prove Theorem~\ref{thm.6-ff} in the $(\ast)$ case. Throughout, we use the term \textit{pro-algebraic} in place of \textit{pro-$\Ical$-algebraic}.\par

To that end, we need to extend the six-functor formalism to pro-algebraic stacks. We will freely use the frameworks developed by Gaitsgory–Rozenblyum, Liu–Zheng, and Mann (cf.~\cite{GR}, \cite{LiuZheng}, \cite{Mann}). Our arguments rely on results from the latter, and for details we refer the reader to \cite[\S A.5]{Mann}. We will not recall the main definitions here, and instead use the language of geometric setups and correspondences as developed in \textit{loc.~cit.}\par

Let $(\Ccal, E)$ be a \textit{geometric setup}, meaning an $\infty$-category $\Ccal$ admitting finite limits, together with a class of morphisms $E \subset \text{Mor}(\Ccal)$ that is stable under pullbacks, compositions, and equivalences. Let
$$
\Dcal \colon \Corr{\Ccal}{E} \to \ICat
$$
be a six-functor formalism in the sense of \cite[\S A.5]{Mann}, where $\Corr{\Ccal}{E}$ denotes the $\infty$-category of correspondences as defined in \cite{LiuZheng}.

We denote by
$$
\Dcal_{\Ccal} \colon \Ccal^{\op} \to \CAlg(\ICat)
$$
the restriction of $\Dcal$ along the canonical inclusion $\Ccal^{\op} \hookrightarrow \Corr{\Ccal}{E}$. We assume that $\Dcal_{\Ccal}$ factors through $\CAlg(\PrLO_{\omega})$. We will use notation analogous to that of Section~\ref{sec.pull}.

 We denote by $\Dcal_{\Ical}$ the composition
$$
	\Dcal_{\Ical}\colon \Fun(\Ical,\Ccal^{\op})\xrightarrow{\Dcal_{\Ccal}} \Fun(\Ical,\textup{CAlg}(\PrLO_{\omega}))\xrightarrow{\colim}\textup{CAlg}(\PrLO_{\omega}).
$$ 
\par

  We define $E_{\Ical}$ to be the class of those morphisms $f\colon X\rightarrow Y$ in $\Fun(\Ical,\Ccal^{\op})$ such that $f_{i}$ is in $E$ for all $i\in \Ical$ and for all $i\rightarrow j\in \Ical$ the square
$$
\begin{tikzcd}
	X_{j}\arrow[r,"X_{ij}"]\arrow[d,"f_{j}"]& X_{i}\arrow[d,"f_{i}"]\\
	Y_{j}\arrow[r,"Y_{ij}"]&Y_{i}
\end{tikzcd}
$$
is cartesian.

\begin{prop}
\label{prop.six.ff}
	Assume that $\Ical$ admits an initial object $0$ and $E$ admits a suitable decomposition $J,F$ that satisfies the assumptions of \cite[Prop. A.5.10]{Mann}.
	Then $\Dcal_{\Ical}$ can be upgraded to a $6$-functor formalism
	$$
		\Dcal_{\Ical}\colon \Corr{\Fun(\Ical,\Ccal^{\op})}{E_{\Ical}}\rightarrow \textup{CAlg}(\PrLO_{\omega}).
	$$
	Further, if $f\in E$ is a morphism such that for all $i\in \Ical$ we have $f_{i}\in F$ resp. $f(i)\in J$, then $f_{!}\simeq f_{*}$ resp. $f_{!}$ is left adjoint to $f^{*}$.
\end{prop}
\begin{proof}
	We define $J_{\Ical}$ resp. $F_{\Ical}$ to be the class of those morphisms $f$ in $E_{\Ical}$ such that $f(i)$ is in $J$ resp. $F$ for all $i\in \Ical$. Then we claim that  $J_{\Ical},F_{\Ical}$ is a suitable decomposition of $E$.\par 
	Indeed, since for a map $f$ in $E$ all squares of the associated diagram are pullback squares. Thus, all of the properties follow immediately (note that we need that $\Ical$ admits an initial object for the decomposition property).\par
	We will use Proposition \ref{prop.ind.pf} and the criterion in \cite[Proposition. A.5.10]{Mann} to prove this proposition.\par
	By Proposition \ref{prop.ind.pf} the functor $\Dcal_{\Ical}$ defines a pullback formalism on $\Fun(\Ical,\Ccal^{\op})$ with respect to $J_{\Ical}$.\par 
	By construction, the pullback $f^{*}$ along any morphism $f\in \Fun(\Ical,\Ccal^{\op})$ admits a right adjoint. We need to check that for morphisms in $F_{\Ical}$ we have base change and projection formula. But similar to the proof of Proposition \ref{prop.ind.pf}, we can use that any compact object comes from some compact on the $i$-th level, for some $i\in \Ical$ (by \textit{mutas mutandis} of Lemma \ref{lem.compact.gen}). This argument shows that for any pullback square
	$$
		\begin{tikzcd}
			W\arrow[r,"f'"]\arrow[d,"g'"]& X\arrow[d,"g"]\\
			Y\arrow[r,"f"]&Z
		\end{tikzcd}
	$$
	with $g\in J_{\Ical}$ and $f\in F_{\Ical}$ we have $g_{\sharp}f'_{*}\simeq f_{*}g'_{\sharp}$ completing the proof.
\end{proof}

The above proposition a priori only applies if the geometric setup has a suitable decomposition. For schemes this is usually not a problem as in practice we can decompose separated morphisms via compactifications into open immersions and proper maps. For algebraic stacks this only holds locally. Nevertheless, the existence of the $!$-adjunction extends to cartesian morphisms of pro-$\Ical$-algebraic stacks by the base change formula, as seen below.

%\begin{conjecture}
%	\label{conj.1}
%	Let $f\colon (x,X)$ be a monoidal pro-$\Ical$-algebraic stack, then 
%	$$
%		\DM(x,X)\simeq \colim_{i\in\Ical,x^{!}}\DM^{!}(X_{i}).
%	$$
%\end{conjecture}
%
%\begin{rem}
%	To prove Conjecture \ref{conj.1}, one only needs that the purity morphism is functorial in the $\infty$-categorical sense. To be more precise, one needs purity to be a $1$-morphism of the functors $f^{*},f^{!}\colon (\Sch_{S}^{\textup{lft}})^{\op}\rightarrow \DGCat$. We think, that this is true following for example the construction of Zavyalov (cf. \cite{Zav}) but we did not find a written proof.
%\end{rem}

\begin{prop}
\label{prop.cart.ext}
Let $f\colon X\rightarrow Y$ be a morphism inside $E_{\Ical}\subseteq \Fun(\Ical^{\op},\Ccal)^{\op}$. Then there exists an adjunction 
 $$
  \begin{tikzcd}
 	f_{!}\colon \Dcal_{\Ical}(X)\arrow[r,"",shift left = 0.3em]&\arrow[l,"",shift left = 0.3em]\Dcal_{\Ical}(Y)\colon f^{!}
 \end{tikzcd}
 $$
satisfying base change and projection formula with respect to $f^{*}\coloneqq\Dcal_{\Ical}(f)$.
\end{prop}
\begin{proof}
First, let us note that the morphism $f$ corresponds to a diagram $\Ical^{\op}\times \Delta^{1,\op}\rightarrow \Ccal$. 
The idea of the proof is to take $f_{!}\coloneqq\colim_{\Ical} f_{!\bullet}$. However, as $\Ical$ is arbitrary, we do not immediately obtain a a diagram $f_{!\bullet}\colon \Ical\times \Delta^{1,\op}\rightarrow \DGCat$ as in the proof of the $(\ast)$ case. Instead, we will make use of the fact that $f$ is cartesian and we will construct a diagram $\Ical\times \Delta^{1,\op}\rightarrow \Corr{\Ccal}{E}$. Taking the colimit of this diagram will then yield $f_{!}$.
To make this process homotopy coherent, we will use the description\footnote{We were made aware of this alternative description by Chirantan Chowdhury.} of $\Corr{\Ccal}{E}$ in terms of bisimplicial sets. 

Let us consider the marked simplicial set $(\Ccal,E,\textup{ALL})$, where $E$ is as before and $\textup{ALL}$ denotes the class of all edges in $\Ccal$. By the assignment $([n_{1}],[n_{2}])\mapsto \Hom_{\mathrm{Set}_{\Delta}}(\Delta^{n_{1}}\times\Delta^{n_{2}},\Ccal)$ we obtain a bisimplicial set $\delta_{*}\Ccal$. We denote by $\Ccal^{\mathrm{cart}}_{E}$ the bisimplicial subset such that the vertical edges are contained in $E$ and each square is cartesian (for a rigorous construction see \cite{chowd}). By a similar construction, we obtain a bisimplicial set  $(\Ical^{\op}\times\Delta^{1,\op})^{\textup{ALL}}_{\textup{ALL}}$ associated to the marked simplicial set $(\Ical^{\op}\times\Delta^{1,\op},\textup{ALL},\textup{ALL})$. This process is functorial in marked simplicial sets with marked squares, so we obtain a map $\alpha\colon (\Ical^{\op}\times\Delta^{1,\op})^{\textup{ALL}}_{\textup{ALL}}\rightarrow \Ccal^{\mathrm{cart}}_{E}$. We take opposites of the horizontal arrows, which we see as a functor $\textup{op}^{2}$ and apply this to $\alpha$. Finally, we let $K\mathrm{pt}^{n}$ denote the bisimplicial subset of $\Delta^{(n,n)}$ spanned by edges $(k,l)$ for $0\leq k\leq l\leq n$. Then we define the simplicial set by 
$$
\delta^{*}_{2\nabla}\op^{2}\Ccal^{\mathrm{cart}}_{E}\colon n\mapsto \Hom_{\mathrm{Set}_{2\Delta}}(K\mathrm{pt}^{n},\op^{2}\Ccal^{\mathrm{cart}}_{E}).
$$
By functoriality of the construction, we obtain a diagram 
$$
	\agtilde\colon \delta^{*}_{2\nabla}\op^{2}(\Ical^{\op}\times\Delta^{1,\op})^{\textup{ALL}}_{\textup{ALL}}\rightarrow \delta_{2\nabla}^{*}\op^{2}\Ccal^{\mathrm{cart}}_{E}.
$$
It is not hard to see from the construction that $\delta^{*}_{2\nabla}\op^{2}\Ccal^{\mathrm{cart}}_{E}\cong \Corr{\Ccal}{E}$ as simplicial sets. The simplicial set $ \delta^{*}_{2\nabla}\op^{2}(\Ical^{\op}\times\Delta^{1,\op})^{\textup{ALL}}_{\textup{ALL}}$ is very explicit and we can construct a map of simplicial sets $\Ical\times\Delta^{1,\op}\rightarrow \delta^{*}_{2\nabla}\op^{2}(\Ical^{\op}\times\Delta^{1,\op})^{\textup{ALL}}_{\textup{ALL}}$ in the following way. 

Note that we have a map
$$
\Hom_{\mathrm{Set}_{\Delta}}(\Delta^{n}\times\Delta^{n,\op},\Ical^{\op}\times\Delta^{1,\op})\rightarrow (\delta^{*}_{2\nabla}\op^{2}(\Ical^{\op}\times\Delta^{1,\op})^{\textup{ALL}}_{\textup{ALL}})_{n}.
$$
Let $\Delta^{n}\to \Ical\times\Delta^{1,\op}$ be a map of simplicial sets. This map is uniquely determined by projections to $\Ical$ and $\Delta^{1,\op}$. Let us denote these maps by $p_{1}\colon\Delta^{n}\to\Ical$ and $p_{2}\colon \Delta^{n}\to  \Delta^{1,\op}$. From these two maps, we get a  map $\Delta^{n}\times\Delta^{n,\op}\to\Ical^{\op}\times\Delta^{1,\op}$ of simplicial sets via 
\begin{align*}
	\varphi_{1}&\colon \Delta^{n}\times\Delta^{n,\op}\xrightarrow{\mathrm{pr}_{2}} \Delta^{n,\op}\xrightarrow{p_{1}^{\op}}\Ical^{\op},\\
	\varphi_{2}&\colon \Delta^{n}\times\Delta^{n,\op}\xrightarrow{\mathrm{pr}_{1}} \Delta^{n}\xrightarrow{p_{2}}\Delta^{1,\op},
\end{align*}
where $\mathrm{pr_{i}}$ denotes the $i$-th projection. In particular, we thus obtain a morphism of simplicial sets 
$$
	\phi\colon \Ical\times\Delta^{1,\op}\rightarrow \delta^{*}_{2\nabla}\op^{2}(\Ical^{\op}\times\Delta^{1,\op})^{\textup{ALL}}_{\textup{ALL}}\rightarrow \Corr{\Ccal}{E}.
$$

Let $\widetilde{\Dcal}$ denote the composition 
$$
	\Fun(\Ical,\Corr{\Ccal}{E})^{\op}\xrightarrow{(\Dcal\circ -)^{\op}} \Fun(\Ical,\textup{CAlg}(\PrLO_{\omega}))^{\op}\xrightarrow{\colim}\textup{CAlg}(\PrLO_{\omega}).
$$
By design, we have $\widetilde{\Dcal}(X) \simeq \Dcal_{\Ical}(X)$, where we view $X$ as an objet of the left hand side via the map $\Fun(\Ical,\Ccal^{\op})^{\op}\to\Fun(\Ical,\Corr{\Ccal}{E})^{\op}$. In particular, by applying $\widetilde{\Dcal}$ to $\phi$, which we view as an edge in $\Fun(\Ical,\Corr{\Ccal}{E})^{\op}$, we obtain a colimit preserving map that preserves compact objects 
$$
	f_{!}\coloneqq \widetilde{\Dcal}(\phi)\colon \Dcal_{\Ical}(X)\rightarrow \Dcal_{\Ical}(Y). 
$$
The remaining assertions follow by Lemma \ref{lem.compact.gen} similar to the proof in the $(\dagger)$ case.
\end{proof}

\begin{rem}
	In the proof of Proposition~\ref{prop.cart.ext}, we can apply the same method to obtain a diagram $\Ical \times \Delta^{1} \to \Corr{\Ccal}{E}$ 
by working with the projection $	\Delta^{n} \times \Delta^{n,\op} \to \Delta^{n,\op}.$ The colimit along this map then yields the pullback functor $f^{*}$.
\end{rem}

\begin{rem}
	As is evident from the proof of Proposition~\ref{prop.cart.ext}, if $\Dcal^{!}_{|\Ccal_{E}}$ commutes with colimits, then so does 
	$$
	f^{!} \coloneqq \Dcal_{\Ical}^{!}(f)
	$$ 
	for all $f \in E_{\Ical}$. This follows from the fact that $\Dcal_{\Ical}(X)$ is compactly generated for all $X \in \Fun(\Ical, \Ccal)$, by Lemma~\ref{lem.compact.gen}.
\end{rem}

\begin{proof}[Proof of Theorem \ref{thm.6-ff}]
	This is now just a consequence of Proposition \ref{prop.six.ff}, Proposition \ref{prop.cart.ext} and Lemma \ref{lem.M}. Note that if $f$ is a proper cartesian morphism of pro-$\Ical$-algebraic stacks, then the construction of $f_{!}$ in Proposition \ref{prop.six.ff} shows that $f_{!}\simeq f_{*}$, using that on compacts they agree by Lemma \ref{lem.compact.gen}.
\end{proof}

\begin{rem}
\label{rem.comp.of.mot}
	Let $f\colon (x,X)\rightarrow (y,Y)$ be a smooth morphism pro-$\Ical$-algebraic stacks. For $M\in\DM(y,Y)$, we can compute the "motivic global sections with values in $M$" directly. To be more precise, let us note that since $f_{*}f^{*}$ is colimit preserving, its values are determined by its restriction to compacts. As any compact in $\DM(y,Y)$ comes from a compact in some $\DM(Y_{i})$ (cf. Lemma \ref{lem.compact.gen}), we see that
	$$
		f_{*}f^{*}M\simeq \colim_{\ins_{i}M_{i}\rightarrow M} \colim_{i\rightarrow k, j\rightarrow k}\ins_{i}^{Y} f_{i*}f^{*}_{i}y_{ik*}y_{ij}^{*}M_{i},
	$$
	where the colimit is over all $i\in \Ical$ and all compacts $M_{i}\in \DM(Y_{i})^{c}$. A similar formula can also be given for $f_{!}f^{!}$.\par
	In particular, if for example $(y,Y)\cong S$, and $(x,X)$ is strict, we have 
	$$
		f_{*}f^{*}1_{S}\simeq f_{0*}f_{0}^{*}1_{S}
	$$
	as expected (cf. Remark \ref{rem.underlying.motive}).
\end{rem}

\begin{rem}[Non-finite type base change]
\label{rem.inf.bc}
Let $f\colon (x,X)\rightarrow (y,Y)$ be a smooth adjointable morphism pro-algebraic stacks. Assume that $f$ is either $\Ical\simeq\NN_{0}$ or that $f$ is cartesian and $\Ical$ has an initial object $0$.\par 
Our computations have shown that we have $f_{*}\ins_{0}^{X}\simeq \ins_{0}^{Y} f_{0*}$, where $\ins_{0}^{X}$ denotes the natural functor $\DM(X_{0})\rightarrow \DM(x,X)$ (similarly for $Y$). This can be seen as a form of non-finite type base change equivalence of the square
$$
	\begin{tikzcd}
		(x,X)\arrow[r,""]\arrow[d,""]& X_{0}\arrow[d,""]\\
		(y,Y)\arrow[r,""]&Y_{0},
	\end{tikzcd}
$$
which is cartesian if $f$ is cartesian. However, note that as longs as $(x,X)\rightarrow X_{0}$ and $(y,Y)\rightarrow Y_{0}$ are not adjointable, we cannot use the base change of Theorem \ref{thm.6-ff}.
\end{rem}

\section{Motivic cohomology of pro-algebraic stacks}

In this section, we want to define and highlight some properties of motivic cohomology in our setting.\par
In the following, we fix a morphism $f\colon (x,X)\rightarrow (y,Y)$ of pro-$\Ical$-algebraic stacks over $S$. Let $h\colon (x,X)\rightarrow S$ and $g\colon (y,Y)\rightarrow S$ be the morphism induced by the structure morphisms. We further assume that either $f$ is cartesian or that $\Ical\simeq \NN$ and $f$ is adjointable.\par 
\par 
By Theorem \ref{thm.6-ff}, we have adjunctions $f^{*}\dashv f_{*}$ and $f_{!}\dashv f^{!}$ between $\DM(x,X)$ and $\DM(y,Y)$ satisfying all the properties stated in the theorem. Note that even though $(x,X)$ and $(y,Y)$ are not assumed to be strict, we still have adjunctions $h^{*}\dashv h_{*}$ and $g^{*}\dashv g_{*}$ by Lemma \ref{lem.Gai.colim.2}.

\begin{defi}
\label{def.mot.coh}
	Let $f\colon (x,X)\rightarrow (y,Y)$ and $g\colon (y,Y)\rightarrow S$ be as above. Then we define the \textit{relative motivic cohomology of $(x,X)$ with coefficients in $M\in \DM(y,Y)$} as 
	$$
		R\Gamma(X,M)\coloneqq \Homline_{\Dcal(\QQ)}(f_{!}f^{!}1_{(y,Y)},M)\in \DGCat. 
	$$\par 
	Further, we define for $n,m\in \ZZ$ the \textit{absolute motivic cohomology of $(y,Y)$ in degree $(n,m)$} as the $\QQ$-vector space 
	$$
		H^{n,m}(Y,\QQ) \coloneqq \pi_{0}\Hom_{\DM(y,Y)}(1_{(y,Y)},1_{(y,Y)}(n)[m]).
	$$
\end{defi}

\begin{rem}
\label{rem.comp.coh}
	Let us remark that the absolute motivic cohomology of $(y,Y)$ can also be computed relative to $S$ in the following sense
	$$
		H^{n,m}(Y,\QQ)\cong  \Hom_{h\DM(S)}(1_{S},g_{*}g^{*}1_{S}(n)[m]).
	$$
	Further, if $(y,Y)$ is strict and $(y,Y)\rightarrow S$ is a smooth, then 
	$$
		H^{n,m}(Y,\QQ) \cong \pi_{m}\Hom_{\DM(S)}(g_{\sharp}g^{*}1_{S},1_{S}(n)).
	$$
\end{rem}

\begin{assumption}
From now on let us assume that $\Ical$ admits an initial object $0\in \Ical$
\end{assumption}

\begin{lem}
	\label{lem.coh.colim}
	Under the notation above, we have
	$$
		g_{*}1_{(y,Y)}\simeq \colim_{j\in \Ical}g_{j*}1_{Y_{j}}.
	$$
\end{lem}
\begin{proof}
	By construction, we have $g^{*}\simeq \ins_{0}^{Y}g^{*}_{0}$. Using adjunctions this shows that 
	$$
		g_{*}\simeq g_{0*}p^{Y}_{0}h_{Y}.
	$$
	By construction, we have $\ins^{Y}_{0}1_{Y_{0}}\simeq 1_{(y,Y)}$. Thus, 
	\begin{align*}
		g_{*}1_{(y,Y)}\simeq g_{0*}p^{Y}_{0}\gamma_{0}^{Y}1_{Y_{0}}
		&\simeq\colim_{j\in \Ical}g_{0*}y_{0j*}y^{*}_{0j}1_{Y_{0}}\\
		&\simeq \colim_{j\in \Ical}g_{j*}1_{Y_{j}}.
	\end{align*}
\end{proof}

\begin{rem}
\label{rem.coh}
	Lemma \ref{lem.coh.colim} immediately shows that for $S$ regular and $(y,Y)\rightarrow S$ smooth, we have 
	$$
		H^{n,m}(Y,\QQ)\cong \colim_{i\in \Ical} A^{n}(Y_{i},2n-m)_{\QQ}.
	$$
\end{rem}
We will see below, that if $(y,Y)$ is classical and $Y\rightarrow S$ is locally of finite type, then this colimit presentation of the absolute motivic cohomology of $Y$ agrees with the classical motivic cohomology of $Y$. In this setting, we can also look at the functor induced by the usual $*$-pushforward on Artin-stacks $\DM(y,Y)\simeq \DM(Y)\rightarrow \DM(S)$. This agrees by construction with $g_{*}\colon \DM(y,Y)\rightarrow \DM(S)$.

\begin{prop}
\label{prop.classical.coh}
	Assume $(y,Y)$ is classical such that $g\colon Y\rightarrow S$ is locally of finite type. Then 
	$$
		H^{n,m}(Y,\QQ)\simeq \Hom_{h\DM(Y)}(1_{Y},1_{Y}(n)[m]).
	$$
\end{prop}
\begin{proof}
	Let $\Ytilde_{0}\rightarrow Y_{0}$ be a smooth cover of $Y_{0}$ by a scheme locally of finite type over $S$. Let $\Ytilde_{i}$ denote the base change to $Y_{i}$ and $\Ytilde$ its limit. Then $\Ytilde$ is representable by a scheme and yields a smooth cover of $Y$. As $Y$ is locally of finite type, so is $\Ytilde$. For every $i\in \Ical$, we can compute 
	$$
		g_{i*}1_{Y_{i}}\simeq \colim_{\Delta}g_{\Cv(\Ytilde_{i}/Y_{i})_{\bullet}*}1_{\Cv(\Ytilde_{i}/Y_{i})_{\bullet}}
	$$
	by descent, where $g_{\Cv(\Ytilde_{i}/Y_{i})_{\bullet}}\colon \Cv(\Ytilde_{i}/Y_{i})_{\bullet}\rightarrow S$ denotes the projection. In particular, we can write 
	$$
		f_{*}1_{(y,Y)}\simeq  \colim_{j\in \Ical}g_{j*}1_{Y_{j}}\simeq \colim_{\Delta}\colim_{j\in \Ical}g_{\Cv(\Ytilde_{i}/Y_{i})_{\bullet}*}1_{\Cv(\Ytilde_{i}/Y_{i})_{\bullet}}
	$$
	(cf. Lemma \ref{lem.coh.colim}).
	Therefore, by continuity of $\DM$ \cite[Thm. 14.3.1]{CD} and descent, we have
	\begin{align*}
		H^{n,m}(Y,\QQ) &\cong \colim_{\Delta}\colim_{i\in \Ical}\Hom_{h\DM(\Cv(\Ytilde_{i}/Y_{i})_{\bullet})}(1_{\Cv(\Ytilde_{i}/Y_{i})_{\bullet}},1_{\Cv(\Ytilde_{i}/Y_{i})_{\bullet}}(n)[m])\\ 
		&\cong \colim_{\Delta}\Hom_{h\DM(\Cv(\Ytilde/Y)_{\bullet})}(1_{\Cv(\Ytilde/Y)_{\bullet}},1_{\Cv(\Ytilde/Y)_{\bullet}}(n)[m])\\ 
		&\cong \Hom_{h\DM(Y)}(1_{Y},1_{Y}(n)[m]).
	\end{align*}
\end{proof}

\begin{example}
\label{ex.field}
	Let us consider the situation of Example \ref{ex.alg.ext}. So let $L/K$ be an algebraic extension of fields and consider a smooth scheme $Z\rightarrow \Spec(K)$. We denote by $\Fcal^{L}_{K}$ the set of finite field extensions $E/K$ contained in $L$. We denote by $(\iota,\Spec(L))$ the associated classical  pro-$\Fcal^{L}_{K}$-algebraic stack over $K$ and by $(z,Z_{L})$ the induced classical pro-$\Fcal^{L}_{K}$-algebraic stack over $K$. Using Remark \ref{rem.comp.coh}, we have
	$$
		H^{n,m}(Z_{L},\QQ) \cong \colim_{E\in \Fcal^{L}_{K}} A^{n}(Z_{E},2n-m)_{\QQ}\cong A^{n}(Z_{L},2n-m)_{\QQ},
	$$
	where the second isomorphism follows from direct computations on the level of cycles \cite[Lem. 1.4.6 (i)]{BroThes}.
\end{example}

\begin{example}
\label{ex.gen.fiber}
	Assume that $S$ is regular. Let $Z\rightarrow S$ be a smooth integral $S$-scheme with function field $K$. Let us consider the family $(U)_{U\in \Ucal_{K}}$ of all affine open subschemes of $Z$. If we denote the respective inclusions by $\iota_{UU'}\colon U'\hookrightarrow U$, then this family assembles to a classical pro-$\Ucal_{K}$-algebraic stack $(\iota,\Spec(K))$. Let $T\rightarrow Z$ be a flat morphism of $S$-schemes. Then base change of $T$ along the system $(\iota,\Spec(K))$ induces a classical pro-$\Ucal_{K}$-algebraic stack $(\iota_{T},T_{K})$. Using Remark \ref{rem.comp.coh}, we have
	$$
		H^{n,m}(T_{K},\QQ) \cong \colim_{U} A^{n}(T_{U},2n-m)_{\QQ}\cong A^{n}(T_{K},2n-m)_{\QQ},
	$$
	where the second isomorphism follows from direct computations on the level of cycles \cite[Lem. 1.4.6 (ii)]{BroThes}.
\end{example}

\subsection{Motives on $\Disp$}

\label{sec.disp}

Let us fix a prime $p>0$ and assume from now on that $S=\Spec(\FF_{p})$. We will denote the Frobenius with $\sigma$. In the following section, we want to define the pro-algebraic stack of (truncated) Displays over $S$, after the construction of Lau (cf. \cite{Lau}) and show that its underlying motive in $\DM(S)$ is Tate. We do not want to go into the detail of the construction of (truncated) Displays and its relation to the stack of Barsotti-Tate groups, but rather give an equivalent definition following \cite{Bro1}.\par
Let us consider the functor $\Spec(R)\mapsto \GL_{h}(W_{n}(R))$ for any affine $S$-scheme $\Spec(R)$, where $W_{n}$ denotes the ring of $n$-truncated Witt-vectors. This functor is represented by an open subscheme of $\AA^{nh^{2}}$, which we  denote by $X_{n}^{h}$. Let $G^{h,d}_{n}(R)$ denote the group of invertible $h\times h$ matricies 
	\begin{equation*}
		\begin{pmatrix}
			A & B   \\
			C & D 
		\end{pmatrix}
	\end{equation*}
with $A\in \GL_{h-d}(W_{n}(R))$, $B\in \End(W_{n}(R)^{d},W_{n}(R)^{h-d})$, $C\in  \End(W_{n}(R)^{h-d},I_{n+1}(R)^{d})$, where $I_{n+1}(R)$ denotes the image of the Verschiebung $W_{n+1}\rightarrow W_{n}$ and $D\in \GL_{d}(W_{n}(R))$.
Let us denote by $G^{h,d}_{n}$ the group scheme representing this functor. Then $G^{h,d}_{n}$ acts on $X^{h}_{n}$ via $M\cdot x\coloneqq Mx \sigma'(M)^{-1}$,
where 
$$
	\sigma'(M)\coloneqq \begin{pmatrix}
			\sigma A & p\sigma B   \\
			\sigma C & \sigma D
		\end{pmatrix}.
$$

\begin{defi}
	For $0\leq d\leq h$ and $n\in\NN$, we define the stack of $n$-truncated displays of dimension $d$ and height $n$ to be 
	$$
		\Disp_{n}^{h,d}\coloneqq \quot{X^{h}_{n}/G^{h,d}_{n}}.
	$$
	Further, we define the stack of $n$-truncated displays as 
	$$
		\Disp_{n}\coloneqq \coprod_{0\leq d\leq h} \Disp_{n}^{h,d}.
	$$
\end{defi}

\begin{notation}
	For each $n\in\NN$, there exists a truncation map $\tau_{n}\colon \Disp_{n+1}\rightarrow \Disp_{n}$. We denote by $(\tau,\Disp)$ the pro-algebraic stack induced by the $\tau_{n}$. 
\end{notation}

\begin{lem}[\protect{\cite[Lem. 2.2.2]{Bro1}}]
\label{lem.unip}
	Let $K^{h,d}_{n,m}$ denote the kernel of the projection $G^{h,d}_{n}\rightarrow  G^{h,d}_{m}$ for $m<n$ and $\Ktilde^{h,d}_{n}$ the kernel of the projection $G^{h,d}_{n}\rightarrow \GL_{h-d}\times \GL_{d}$. Then $K^{h,d}_{n,m}$ and $\Ktilde^{h,d}_{n}$ are split unipotent.
\end{lem}

\begin{rem}
\label{rem.Bro}
	As seen in the proof of \cite[Thm. 2.3.3]{Bro1} the group $G_{1}^{h,d}$ is even a split extension of $\GL_{h-d}\times \GL_{d}$ by a split unipotent group $U_{1}$. The splitting is induced by the canonical inclusion $\GL_{h-d}\times \GL_{d}\hookrightarrow G^{h,d}_{1}$.
\end{rem}

\begin{prop}
\label{prop.Disp.strict}
	The pro-algebraic stack $(\tau,\Disp)$ is strict and tame.
\end{prop}
\begin{proof}
	As $\DM$ preserves finite products, it is enough to prove that for any $0\leq d\leq h$ and any $n>m$ the truncation $\tau^{h,d}_{n,m}\colon \Disp^{h,d}_{n}\rightarrow \Disp^{h,d}_{m}$ is strict.
	By definition $\tau_{n,m}$ factors through 
	$$
		\quot{X^{h}_{n}/G^{h,d}_{n}}\xrightarrow{a} \quot{X^{h}_{m}/G^{h,d}_{n}}\xrightarrow{b} \quot{X^{h}_{m}/G^{h,d}_{m}}.
	$$
	The map $b^{*}\colon \DM(\quot{X^{h}_{m}/G^{h,d}_{m}})\rightarrow  \DM(\quot{X^{h}_{m}/G^{h,d}_{n}})$ is an equivalence of categories by Lemma \ref{lem.unip} (cf. \cite[Prop. 2.2.11]{RS1}).\par 
	We now claim that $a^{*}$ is fully faithful (the idea is the same as in \cite[Thm 2.3.1]{Bro1}). By definition the map $X^{h}_{n}\rightarrow X^{h}_{m}$ is a $K^{h,0}_{n,m}$-torsor. By \'etale descent we may replace the morphism $a$ by the projection $K^{h,0}_{n,m}\rightarrow S$. By Lemma \ref{lem.unip} this group is split unipotent, so it admits a normal series with successive quotients given by a vector bundle. Thus, applying again \'etale descent and induction, we may assume that  $K^{h,0}_{n,m}$ is a vector bundle over $S$, showing that $a^{*}$ is fully faithful.
	
	The factorization also shows that $\tau$ is tame.
\end{proof}

\begin{lem}[\protect{\cite{LauG}}]
\label{lem.Disp.pres}
	The functors $X^{h,d}_{\infty}\coloneqq \lim_{n} X^{h,d}_{n}$ and $G^{h,d}_{\infty}\coloneqq\lim_{n}G^{h,d}_{n}$ are representable by affine $\FF_{p}$-group schemes. Moreover, $G^{h,d}_{\infty}$ acts on $X^{h,d}_{\infty}$ and we have an equivalence of algebraic stacks $\quot{X^{h,d}_{\infty}/G^{h,d}_{\infty}}\cong \Disp$. 
\end{lem}
%\begin{proof}
%	The representability follows from \cite[Lem. 5.4.1]{LauG}.\footnote{We can use Lau's results in the case of $G=\GL_{h}$, $\mu=(1,\dots,1,0\dots,0)$, where the $1$'s appear $h-d$ and $0$'s appear $d$ times and consider the Witt frame described in \cite[Ex. 2.1.3]{LauG}.} The action is described in \cite[\S 5.1]{LauG} and the representation of $\Disp$ as the quotient of this action follows from the definitions (see also \cite[Rem. 5.4.3]{LauG}). 
%\end{proof}

\begin{prop}
\label{prop.Disp.motive}
	The natural morphism $\DM(\tau,\Disp)\rightarrow \DM(\Disp)$ induces an equivalence of DG-categories.
\end{prop}
\begin{proof}
	Recall that $\Disp_{n}^{h,d}\cong\quot{X_{n}^{h,d}/G^{h,d}_{n}}$. Let us consider the notation of Lemma \ref{lem.Disp.pres}. The pro-algebraic group scheme $G^{h,d}_{\infty}$ acts on each $X_{n}^{h,d}$ via restriction for each $n\in \NN$.
	 
	Note that we have for any $1\leq n\leq m\leq \infty$ a cartesian diagram
	\begin{equation}
	\label{eq.bc}
		\begin{tikzcd}
			X_{m}^{h,d}\arrow[r,""]\arrow[d,""]& X_{n}^{h,d}\arrow[d,""]\\
			\quot{X_{m}^{h,d}/G^{h,d}_{\infty}}\arrow[r,""]&\quot{X_{n}^{h,d}/G^{h,d}_{\infty}}.
		\end{tikzcd}
	\end{equation}
  \cite[Cor. A.13]{WedG}. In particular, the induced map 
	$$
	\lim_{n}X_{n}^{h,d}\cong X_{\infty}^{h,d}\rightarrow \quot{X_{\infty}^{h,d}/G^{h,d}_{\infty}}
	$$
	is an effective epimorphism of \'etale stacks and the associated \v{C}ech nerve is given by base change of the \v{C}ech nerve of $X_{1}^{h,d}\rightarrow \quot{X_{1}^{h,d}/G^{h,d}_{\infty}}$.

	Each transition map $\quot{X_{n+1}^{h,d}/G^{h,d}_{\infty}}\rightarrow \quot{X_{n}^{h,d}/G^{h,d}_{\infty}}$ is an affine bundle as explained in the proof of Proposition \ref{prop.Disp.strict}. We also note that $G^{h,d}_{\infty}$ is an affine faithfully flat group scheme.\par
	
	We have
	$$
	\begin{tikzcd}[sep=small]
		 \DM([X_{n}^{h,d}/G^{h,d}_{\infty}])\simeq \lim_{n}(\DM(X_{n}^{h,d})\arrow[r,"",shift right = 0.3 em]\arrow[r,"",shift left = 0.3 em]&\DM(X_{n}^{h,d}\times_{\FF_{p}} G^{h,d}_{\infty})\arrow[r,""]\arrow[r,"",shift right = 0.5 em]\arrow[r,"",shift left = 0.5 em]&...),
	\end{tikzcd}
	$$
	for all $1\leq n\leq \infty$ \cite[Thm. 2.2.16]{RS1}. Having the cartesian diagram (\ref{eq.bc}) in mind, we can now repeat the proof of Proposition \ref{prop.pres.colim} and obtain
	$$	
		\DM(\quot{X_{\infty}^{h,d}/G^{h,d}_{\infty}})\simeq \colim_{n} \DM(\quot{X_{n}^{h,d}/G^{h,d}_{\infty}}).
	$$
	Finally, by \cite[Prop. 2.2.11]{RS1}, we have 
	$$
	\DM(\quot{X_{n}^{h,d}/G^{h,d}_{\infty}})\simeq  \DM(\quot{X_{n}^{h,d}/G^{h,d}_{n}})
	$$
	concluding the proof.
\end{proof}

\begin{thm}
\label{thm-Disp}
	The pullback along the truncation map $\Disp\rightarrow \Disp_{1}$ induces a fully faithful embedding.
	$$
		\DM(\Disp_{1})\hookrightarrow \DM(\Disp)
	$$
	In particular, we have an equivalence of motives 
	$$
		M(\Disp)\simeq M(\Disp_{1})
	$$
	inside $\DM(\FF_{p})$ and moreover it is contained in the full stable cocomplete subcategory of $\DM(\FF_{p})$ generated by Tate motives. 
\end{thm} 
\begin{proof}
	By  Proposition \ref{prop.Disp.strict} the natural map $\DM(\Disp_{1})\rightarrow \DM(\tau,\Disp)$ is fully faithful. By Proposition \ref{prop.Disp.motive} the natural map $\DM(\tau,\Disp)\rightarrow \DM(\Disp)$ is an equivalence. This proves that $\DM(\Disp_{1})\rightarrow \DM(\Disp)$ is fully faithful and $M(\Disp)\simeq M(\Disp_{1})$. The fact, that this motive is Tate follows\footnote{Note that $\Disp_{1}$ is isomorphic to the stack of $F$-zips, considered in \cite{yay23}.} from \cite{yay23}. 
\end{proof}

\begin{rem}
	Brokemper computes the Chow groups of $\Disp_{1}^{h,d}$ explicitly in \cite[Thm. 2.3.3]{Bro1}. Thus, Theorem \ref{thm-Disp} and Remark \ref{rem.underlying.motive} tell us that the motivic cohomology of $\Disp$ can be computed by the Chow groups of $\Disp_{1}^{h,d}$. To be more precise, we can compute
	$$
		\bigoplus_{n\in \ZZ}H^{n,2n}(\Disp,\QQ)\cong \bigoplus_{0\leq d\leq h}\QQ[t_{1},\dots,t_{h}]^{S_{h}\times S_{h-d}}/(c_{1},\dots,c_{h}),
	$$
	where $c_{i}$ denotes the $i$-th elementary symmetric polynomial in variables $t_{1},\dots,t_{h}$.
\end{rem}

\begin{rem}
\label{rem.BT}
	Let $\BT$ be the stack of Barsotti-Tate groups (short, BT-groups) over $S$. For any $n\in \NN$, we denote by $\BT_{n}$ the stack of level-$n$ BT-groups. For any $n\geq 0$ there exists a truncation map $\tau_{n}\colon \BT_{n+1}\rightarrow \BT_{n}$ and $\lim_{n,\tau} \BT_{n} \simeq \BT$. The maps $\tau_{n}$ are smooth (cf. \cite[Thm. 4.4]{Ill}). Note that every level-$n$ BT-group admits the notion of a height and dimension, which are locally constant functions over $S$. In particular, we can write $\BT_{n} \simeq \coprod_{0\leq d\leq h} \BT^{h,d}_{n}$, where $\BT^{h,d}_{n}$ denotes the substack of $\BT$ generated by BT-groups of dimension $d$ and height $h$. \par 
We want to remark, that there is a morphism $\phi\colon \BT\rightarrow \Disp$ compatible with truncations (cf. \cite{Lau}). Moreover, each of the maps $\phi_{n}$ is smooth and an equivalence on geometric points (cf. \cite[Thm. A]{Lau}). In particular, $\phi_{n}$ is a universal homeomorphism. Note however, that this is not enough to see that $\DM(\tau,\Disp)\simeq\DM(\tau,\BT)$ as the morphisms $\phi_{n}$ are not representable. We, conjecture that $\phi^{*}\colon \DM(\Disp)\rightarrow \DM(\BT)$ restricts to Tate-motives proving that the underlying motive of $\BT$ and $\Disp$ are equivalent. However, as pro-algebraic stacks the motivic cohomology groups of $\BT$ and $\Disp$ are isomorphic.
\end{rem}

\begin{cor}
\label{cor.BT}
	We have an isomorphism of motivic cohomology groups 
	$$
		H^{*,*}(\BT,\QQ)\cong H^{*,*}(\Disp,\QQ)\cong H^{*,*}(\Disp_{1},\QQ).
	$$
\end{cor}
\begin{proof}
	This follows form the definition of the motivic cohomology for the pro-algebraic stack $(\tau,\BT)$, Theorem \ref{thm-Disp} and \cite[Thm. 2.5.4]{Bro1}.
\end{proof}

\subsection{Action of the absolute Galois group on the motivic homology spectrum}

Let us come back to the setting of Example \ref{ex.alg.ext} and fix the setting for this subsection. So as in the example, let $L/K$ be a Galois extension of fields. Let $X$ be a smooth $K$-scheme. We denote by $\Fcal^{L}_{K}$ the filtered category of all finite extensions of $K$ contained in $L$. The inclusion along finite extensions $E\hookrightarrow F$ of $K$ yield pro-$\Fcal^{L}_{K}$-algebraic stacks $(x,X_{L})$ and $(q,L)$. The base change of the structure map $f\colon X\rightarrow \Spec(K)$ to $E\in\Fcal^{L}_{K}$ yields a map of diagrams 
$$
	f_{\bullet}\colon (x,X_{L})\rightarrow (q,L)
$$\par
 Let $\Gcal^{L}_{K}$ denote the filtered subcategory of all finite Galois extensions of $K$. Then $\Gcal^{L}_{K}\subseteq \Fcal^{L}_{K}$ is cofinal. In particular, we see that 
$$
 	\DM(x,X_{L})\simeq \colim_{E\in \Gcal^{L}_{K},x^{*}} \DM(x,X_{E}).
$$
Thus, to understand the action of $\Gal(L/K)$ on $f_{L*}1_{(x,X_{L})}$, we may restrict $(x,X_{L})$ and $(q,L)$ to $\Gcal^{L}_{K}$ and work with their underlying pro-$\Gcal^{L}_{K}$-algebraic stacks. By abuse of notation, we will keep the notation of $(x,X_{L})$ and $(q,L)$.\par 
We naturally get a map $g\colon(q,\Spec(L))\rightarrow \Spec(K)$, where we view $\Spec(K)$ as a constant diagram. We denote the composition $g\circ f_{\bullet}$ by $h_{\bullet}$.

\begin{rem}
	Note that Proposition \ref{prop.pres.colim} shows that 
	$$
		\DM(x,X_{L})\simeq \DM(X_{L})\textup{ and } \DM(q,L)\simeq \DM(L)
	$$

	We will see in the next lemma that our formalism enables us to compute the motive $f_{L*}1_{(x,X_{L})}$ by "pulling back" $f_{*}1_{X}$ along the natural map $\DM(K)\rightarrow \DM(q,L)$.
\end{rem}

Assume that $L/K$ is a \textit{finite} extension. Then the smooth base change immediately yields 
$$
	f_{L *}1_{(x,X_{L})}\simeq  f_{L*}x_{L/K}^{*}1_{X}\simeq q_{L/K}^{*}f_{*}1_{X}.
$$
If $L/K$ is not finite this does not hold as there is no base change formalism that shows this result. In our formalism we have an analog of a smooth base change in this case as in Remark \ref{rem.inf.bc}.

Next, we claim that $\Gal(L/K)$ acts on $h_{*}1_{(x,X_{L})}$. To see this, it is enough to construct actions of $\Gal(E/K)$ on $h_{*}1_{(x,X_{L})}$ for any finite Galois extension $E/K$ that is compatible with the presentation of $\Gal(L/K)$ as a limit of such.

\begin{rem}
\label{rem.Gal.action}
	Let $E/K$ be a finite Galois extension and $\varphi\in \Aut_{K}(E)$. And $M\in \DM(q,\Spec(L))$. Then $\varphi$ induces an automorphism
	$$
		g_{*}M\rightarrow g_{*}M
	$$
	via the following.\par 
	Let $\Gcal_{E}^{L}\subseteq \Gcal_{K}^{L}$ denote the subcategory of finite Galois extensions over $E$. Then we can restrict $q$ along this (filtered) subcategory and we denote the induced pro-$\Gcal_{E}^{L}$-algebraic stack by $(q_{|E},\Spec(L))$. The $K$-automorphism $\varphi$ induces via base change a cartesian automorphism $\varphi\colon (q_{|E},\Spec(L))\rightarrow (q_{|E},\Spec(L))$. The $*$-pushforward along $\varphi$ yields a functor 
	$$
		\varphi_{*}\colon \DM(q,\Spec(L))\rightarrow  \DM(q,\Spec(L)).
	$$
By construction $g_{*}\varphi_{*}\simeq g_{*}$, yielding the desired endomorphism above.\par This endomorphism by construction induces an action of $\Gal(L/K)$ on $g_{*}M$, i.e. a map $B\Gal(L/K)\rightarrow \DM(K)$.
\end{rem}

\begin{rem}
\label{rem.invariants}
	Combining the above, we have
	$$
		\colim _{E\in\Gcal_{K}^{L}}h_{*}1_{X}^{\Gal(E/K)}\simeq\colim _{E\in\Gcal_{K}^{L}}\colim _{F\in\Gcal_{E}^{L}}q_{F/K*}f_{F*}1_{X_{F}}^{\Gal(E/K)} \simeq f_{*}1_{X},
	$$
where we use \cite[Lem. 2.1.166]{ayoub} in the first equivalence. In particular, for a $X$ smooth $K$-scheme, we have
	$$
		A^{n}(X,m)_{\QQ}\simeq A^{n}(X_{L},m)^{\Gal(L/K)}_{\QQ}
	$$
	by Example \ref{ex.field}. This recovers the computation on the level of cycles \cite[Lem. 1.3.6]{Bro1}.\par 
\par 
	There exists a notion of continuous homotopy fixed points on so-called discrete $G$-spectra \cite{BeDa}. The object 
	$\colim _{E\in\Gcal_{K}^{L}}h_{*}1_{X}^{\Gal(E/K)}$ can be seen as an analog of continuous homotopy fixed points under the action of the absolute Galois group. By construction the action of $\Gal(L/K)$ is stabilized on an open and closed normal subgroup after passage to a large enough field extension.
\end{rem}

%------------------------------------------------------------------

%------------------------------------------------------------------

%==================================================================
\addcontentsline{toc}{section}{References}
\bibliographystyle{alphaurl}
\bibliography{MotOnPro}

\end{document}